\documentclass[11pt,a4paper]{article}

\usepackage[a4paper,margin=21mm,headheight=14pt,footskip=11mm]{geometry}
\usepackage[T1]{fontenc}
\usepackage[utf8]{inputenc}
\usepackage{lmodern}
\usepackage{microtype}
\usepackage{amsmath,amssymb,amsthm,mathtools}
\usepackage{booktabs,array,enumitem}
\usepackage{tabularx}
\usepackage{graphicx,listings,float}
\usepackage[dvipsnames,table]{xcolor}
\usepackage{tikz}
\usetikzlibrary{arrows.meta,calc,positioning,matrix,decorations.pathreplacing}
\usepackage{hyperref}
\definecolor{ClockBlue}{HTML}{315F7D}
\definecolor{ClockGreen}{HTML}{187565}
\definecolor{ClockGold}{HTML}{B47A1F}
\definecolor{ClockRed}{HTML}{A34A43}
\definecolor{SoftBlue}{HTML}{EEF4F7}
\definecolor{SoftGreen}{HTML}{EDF6F3}
\definecolor{SoftGold}{HTML}{FBF4E7}
\definecolor{SoftRed}{HTML}{FAEFED}
\definecolor{Ink}{HTML}{25333A}
\definecolor{PaleGray}{HTML}{F5F6F6}
\definecolor{PrimeGray}{HTML}{ECEEEF}
\definecolor{StepBlue}{RGB}{66,132,176}
\definecolor{StepViolet}{RGB}{119,94,164}
\definecolor{StepRose}{RGB}{183,94,124}

\lstdefinestyle{clockpython}{
  language=Python,
  basicstyle=\ttfamily\scriptsize,
  keywordstyle=\bfseries\color{ClockBlue},
  commentstyle=\itshape\color{ClockGreen},
  stringstyle=\color{ClockRed},
  numberstyle=\tiny\color{Ink!60},
  numbers=left,
  numbersep=6pt,
  showstringspaces=false,
  keepspaces=true,
  columns=fullflexible,
  frame=single,
  framerule=0.55pt,
  rulecolor=\color{ClockBlue!55},
  backgroundcolor=\color{PaleGray},
  xleftmargin=5mm,
  framexleftmargin=3mm,
  aboveskip=3pt,
  belowskip=3pt
}

\lstdefinestyle{clocklisp}{
  language=Lisp,
  basicstyle=\ttfamily\scriptsize,
  keywordstyle=\bfseries\color{ClockBlue},
  commentstyle=\itshape\color{ClockGreen},
  stringstyle=\color{ClockRed},
  numberstyle=\tiny\color{Ink!60},
  numbers=left,
  numbersep=6pt,
  showstringspaces=false,
  keepspaces=true,
  columns=fullflexible,
  frame=single,
  framerule=0.55pt,
  rulecolor=\color{ClockGreen!65},
  backgroundcolor=\color{SoftGreen!45},
  xleftmargin=5mm,
  framexleftmargin=3mm,
  aboveskip=3pt,
  belowskip=3pt
}

\hypersetup{
  colorlinks=true,
  linkcolor=ClockBlue,
  citecolor=ClockGreen,
  urlcolor=ClockBlue,
  pdftitle={The Prime Clockwork: Modular and Multiplicative Arithmetic as a Discrete Dynamical System},
  pdfauthor={Michael T.M. Emmerich}
}

\setlist[itemize]{leftmargin=5.5mm,itemsep=2pt,topsep=3pt}
\setlist[enumerate]{leftmargin=6.5mm,itemsep=2pt,topsep=3pt}
\renewcommand{\arraystretch}{1.14}

\newtheoremstyle{clockplain}%
  {5pt}{5pt}{\itshape}{}%
  {\bfseries\color{ClockGreen}}{.}{0.5em}{}
\theoremstyle{clockplain}
\newtheorem{proposition}{Proposition}
\newtheorem{theorem}[proposition]{Theorem}
\newtheorem{lemma}[proposition]{Lemma}
\newtheorem{corollary}[proposition]{Corollary}
\theoremstyle{definition}
\newtheorem{definition}[proposition]{Definition}
\theoremstyle{remark}
\newtheorem{remark}[proposition]{Remark}

\newcommand{\Q}{\mathbb Q}
\newcommand{\Z}{\mathbb Z}
\newcommand{\N}{\mathbb N}

\newcommand{\Pset}{\mathcal P}
\newcommand{\G}{\mathcal G}
\newcommand{\V}{\mathbf V}
\newcommand{\VQ}{\mathbf V_{\!\Q}}
\newcommand{\vp}{\nu_p}
\newcommand{\lcm}{\operatorname{lcm}}

\title{\textbf{The Prime Clockwork: A Dynamic Representation of Modular and Multiplicative Arithmetic}}
\author{Michael T.M. Emmerich\\[-1mm]
\small Faculty of Information Technology, University of Jyv\"askyl\"a, Finland}
\date{September 2026}

\begin{document}
\maketitle
\vspace{-4mm}

\begin{abstract}
The way numbers are represented strongly influences which arithmetic structures are easy to see. The prime clockwork is a recursively growing discrete dynamical system: a list of autonomous two-hand clocks driven by one common $+1$ signal. No primes or primality labels are supplied. Starting empty, the process appends a clock of period $n$ whenever none already present rings; the primes are generated internally as its growth times. For each installed prime $p$, the seconds reading $R_p$ advances through $0,\ldots,p-1$, and each return to zero increments the minutes reading $M_p$, which counts completed $p$-cycles. The hands use only increment, comparison, reset, and carry, without explicit mod or div operations.

At time $n$, $n=pM_p(n)+R_p(n)$. The valuation readout $V_p(n)=\nu_p(n)$ is generated locally: it is zero when the $p$-clock does not ring; when it rings, the readout is one plus the earlier valuation indexed by the current minutes reading. The valuation trajectory has no repeated row; this reconstructs each integer in unique prime-factorized form. Its coordinates add and subtract under multiplication and division, representing every positive rational uniquely; divisibility becomes weak componentwise order, and unique factorization is natural in this representation. Finite seconds arrays form Cartesian-product state spaces whose common orbit visits every joint state once before repeating; this grand cycle is the order-sensitive dynamical counterpart of the Chinese remainder theorem. The same coordinates expose gcd, lcm, perfect powers, B'ezout's identity, and Euler's totient. Rational valuation levels reach certain positive algebraic irrationalities, but not algebraic numbers in general.
\end{abstract}


\begin{flushright}
\begin{minipage}{0.86\textwidth}
\small\itshape
``We are not trying to meet some abstract production quota of definitions, theorems and proofs. The measure of our success is whether what we do enables people to understand and think more clearly and effectively about mathematics.''
\par\smallskip
\raggedleft\normalfont---William P. Thurston, \emph{On Proof and Progress in Mathematics}~\cite[p.~163]{Thurston1994}
\end{minipage}
\end{flushright}
\vspace{1mm}

\section{Introduction}

Elementary number theory is usually presented through several complementary languages: divisibility, congruences, prime factorization, arithmetic functions, and valuations. The prime clockwork asks whether these familiar structures can be organized in one prime-indexed representation before the usual algebraic machinery is brought fully into view.

The construction is intentionally elementary. It begins with no supplied primes and advances only by a common $+1$ signal. Whenever none of the existing clocks rings, the present integer becomes the period of a newly appended clock; the primes are therefore outputs of the recursion rather than parameters given to it. Each complete $p$-clock autonomously updates two hands without consulting any other prime clock: a cyclic seconds reading and an integer minutes reading that counts completed $p$-cycles. The $p$-adic valuation is a third, derived readout, computed at ring times from the minutes address and that clock's own earlier valuation history. Globally, however, unique factorization makes the complete valuation vector sufficient to recover the represented integer, every local clock state, and the unique successor state. Remark~\ref{rem:global-reconstruction} separates this global determinacy from the direct local updateability of the maintained clock hands. This suggests the language of a discrete dynamical system: there are states, update maps, periods, return times, products of state spaces, hidden variables, and observable events such as a return to a distinguished position.

The motivation is fourfold. First, it is an exploration of number representation: changing coordinates can make different structures visible. Second, it is didactic: the aim is a visual and intuitive route into modular and multiplicative number theory. Third, the clock language suggests alternative proofs of classical facts by using cycles, counting, and state transitions. Fourth, the construction provides a modest bridge to discrete dynamical systems without replacing standard number-theoretic notation.

More specifically, the contribution is a single recursive description in which prime periods are generated rather than supplied, seconds and minutes states are advanced clock by clock without explicitly invoking \texttt{mod} or \texttt{div}, valuations are recovered by local history lookup, their rows are proved non-repeating by descent through minutes addresses, and their coordinates turn divisibility into weak componentwise order and division into subtraction.

The metaphor is kept literal. For each prime $p$, the \emph{seconds counter} advances one position per unit tick and resets after $p$ positions. Every reset carries one unit to the \emph{minutes counter}, which records how many $p$-cycles have been completed. The valuation $\nu_p$ is not the minutes reading itself: it is obtained from the ring event, the current minutes value, and the clock's previously generated valuation readings. The two hands form an autonomous successor system; the valuation is its history-dependent multiplicative observation. This distinction will remain in force throughout the paper.

The exposition proceeds from the basic construction to integers, then to rationals and a limited class of algebraic numbers. The closing sections collect what the representation makes visible and position the viewpoint relative to standard theory and earlier clock- and sieve-based expositions.

\section{The basic clockwork}

The prime clockwork is easiest to understand first as an array of autonomous two-hand clocks. Each clock has its own prime period, but all clocks receive the same unit tick and no clock drives another. Its seconds counter advances by cyclic successor; each return to zero adds one to its minutes counter. The valuation readout is then generated from the ring event, the minutes address, and the clock's own stored valuation history. Thus the two hands update autonomously prime by prime, while the valuation is a local history-dependent output. At any finite stage the maintained seconds state is a point in a Cartesian product of finite cyclic state spaces.

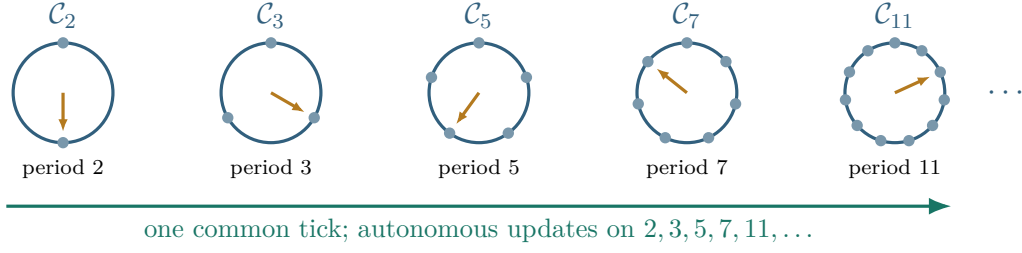
\begin{figure}[ht]
\centering
\begin{tikzpicture}[font=\small,>=Latex]
  \def\R{0.66}
  \foreach \x/\p/\ang in {-5.5/2/270,-2.75/3/-30,0/5/-126,2.75/7/141.43,5.5/11/24.55}{
    \begin{scope}[shift={(\x,0.35)}]
      \draw[ClockBlue,very thick] (0,0) circle[radius=\R];
      \pgfmathtruncatemacro{\last}{\p-1}
      \foreach \j in {0,...,\last}{
        \fill[ClockBlue!65] ({90-360/\p*\j}:\R) circle[radius=.75mm];
      }
      \draw[-{Latex[length=1.7mm]},very thick,ClockGold] (0,0)--(\ang:{0.52});
      \node[font=\bfseries,ClockBlue] at (0,1.02) {$\mathcal C_{\p}$};
      \node[font=\scriptsize] at (0,-1.00) {period $\p$};
    \end{scope}
  }
  \node[font=\large,ClockBlue] at (7.0,0.35) {$\cdots$};
  \draw[-{Latex[length=3mm]},ClockGreen,very thick] (-6.25,-1.15)--(6.25,-1.15);
  \node[ClockGreen] at (0,-1.47) {one common tick; autonomous updates on $2,3,5,7,11,\ldots$};
\end{tikzpicture}
\caption{The autonomous prime-clock array. The clocks $2,3,5,7,11,\ldots$ share a common time parameter, but each updates its seconds and minutes readings independently of the other prime clocks. Only the cyclic seconds hand is displayed; its completed turns are retained by the integer minutes counter.}
\label{fig:autonomous-array}
\end{figure}

For an integer period $p\ge2$, let $\mathcal C_p$ have a seconds counter $R_p\in\{0,1,\ldots,p-1\}$ and an integer minutes counter $M_p\in\mathbb Z_{\ge0}$. Starting from $(R_p(0),M_p(0))=(0,0)$, one unit tick applies the carry update
\begin{equation}
T_p(r,m)=
\begin{cases}
(r+1,m),&r<p-1,\\
(0,m+1),&r=p-1.
\end{cases}
\label{eq:two-hand-update}
\end{equation}
A clock \emph{rings} on the second branch, when its seconds counter returns to $0$. Induction on $n$ gives
\begin{equation}
\boxed{n=pM_p(n)+R_p(n)},
\qquad
R_p(n)=n\bmod p,
\qquad
M_p(n)=\left\lfloor\frac np\right\rfloor.
\label{eq:clock-division}
\end{equation}
The quotient and remainder notation describes the generated readings; the update~\eqref{eq:two-hand-update} itself uses only increment, comparison, reset, and carry. Formally, the full prime-indexed seconds vector has the natural zero initial state
\begin{equation}
\mathbf R(0)=(0,0,\ldots),
\qquad
\mathbf R(n+1)=T(\mathbf R(n)),
\label{eq:seconds-recursion}
\end{equation}
where $T$ advances every active seconds coordinate by one with its own carry rule. The online recursion below realizes this same trajectory lazily: a $p$-clock is instantiated only when $p$ is discovered, in the state $(R_p(p),M_p(p))=(0,1)$ that it would have reached after running from time zero.

The multiplicative reading attached to a prime is the nonnegative integer
\begin{equation}
V_p(n):=\nu_p(n):=\max\{e\in\mathbb Z_{\ge0}:p^e\mid n\},
\label{eq:valuation-readout}
\end{equation}
called the \emph{valuation readout}. We write
\[
\V(n)=(V_2(n),V_3(n),V_5(n),\ldots),
\qquad \V(1)=\mathbf0.
\]
The valuation trajectory is generated locally from the two clock hands. With $V_p(j)=0$ for $1\le j<p$,
\begin{equation}
V_p(n)=
\begin{cases}
0, & R_p(n)\ne0,\\[1mm]
1+V_p(M_p(n)), & R_p(n)=0,
\end{cases}
\qquad n>1.
\label{eq:valuation-recursion}
\end{equation}
Indeed, at a ring equation~\eqref{eq:clock-division} gives $n=pM_p(n)$, so $\nu_p(n)=1+\nu_p(M_p(n))$; away from a ring, $p\nmid n$ and the valuation is zero. Since $M_p(n)<n$, the ring branch refers only to an earlier row. Each $p$-clock can therefore store its own valuation history $H_p(j)=V_p(j)$ and, at a ring, read $H_p(M_p(n))$ and add one. The $+1$ is applied to the value at the minutes address, not to the valuation at the preceding ring. Appendix~\ref{app:python} implements this local lookup directly.

For example, the $3$-clock has the following selected ring states:
\[
\begin{array}{c|cccc}
n&3&6&9&27\\ \hline
(R_3(n),M_3(n))&(0,1)&(0,2)&(0,3)&(0,9)\\
V_3(M_3(n))&0&0&1&2\\
V_3(n)&1&1&2&3
\end{array}
\]
Thus minutes count completed turns in the ordinary clock sense, whereas valuations measure recursively nested divisibility by $p$.

\subsection{The unit-step recursion}
\label{sec:recursive-primes}

The autonomous realization does not begin with a supplied list of primes. At positive time $1$ no clocks are installed. To pass from time $n$ to time $n+1$, apply the following deterministic update.

\begin{enumerate}[label=\arabic*.,leftmargin=8mm]
\item \textbf{Advance the two hands.} Apply~\eqref{eq:two-hand-update} to every installed clock, carrying one unit to $M_p$ whenever $R_p$ returns to $0$.
\item \textbf{Grow the array.} If $n+1>1$ and none of the installed seconds counters is at $0$, install a new clock $\mathcal C_{n+1}$ in state $(R_{n+1},M_{n+1})=(0,1)$.
\item \textbf{Generate valuations.} Each active $p$-clock sets $V_p(n+1)$ from~\eqref{eq:valuation-recursion} using its own stored history; coordinates not yet installed are suppressed dynamically and are retrospectively $0$ in the static valuation vector.
\end{enumerate}

In compact form,
\begin{equation}
\boxed{\text{tick with carry}\;\longrightarrow\;\text{inspect zero alarms}\;\longrightarrow\;\text{grow if none rings}\;\longrightarrow\;\text{read valuations}.}
\label{eq:recursive-install}
\end{equation}
This is the prime clockwork as a recursively growing discrete dynamical system. Each clock autonomously generates its paired seconds and minutes trajectory, and its local valuation history turns those two readings into the multiplicative depth $V_p(n)$. The clock hands are the maintained dynamical state; the valuation is the synchronized derived observation.

The first few steps make the distinction visible. Dots mean that a clock has not yet been installed in the online recursion. Soft gray shading is used only for prime rows, since these are precisely the growth events of the recursion.

\begin{figure}[H]
\centering
\begin{minipage}[t]{0.485\textwidth}
\centering
\textbf{Seconds trajectory $\mathbf R(n)$}\\[1mm]
\scriptsize
\setlength{\tabcolsep}{2.5pt}
\renewcommand{\arraystretch}{1.10}
\begin{tabularx}{\linewidth}{@{}c*{6}{>{\centering\arraybackslash}X}@{}}
\toprule
$n$ & $2$ & $3$ & $5$ & $7$ & $11$ & $13$\\
\midrule
1 & $\cdot$ & $\cdot$ & $\cdot$ & $\cdot$ & $\cdot$ & $\cdot$\\
\rowcolor{PrimeGray}\textbf{2} & 0 & $\cdot$ & $\cdot$ & $\cdot$ & $\cdot$ & $\cdot$\\
\rowcolor{PrimeGray}\textbf{3} & 1 & 0 & $\cdot$ & $\cdot$ & $\cdot$ & $\cdot$\\
4 & 0 & 1 & $\cdot$ & $\cdot$ & $\cdot$ & $\cdot$\\
\rowcolor{PrimeGray}\textbf{5} & 1 & 2 & 0 & $\cdot$ & $\cdot$ & $\cdot$\\
\addlinespace[1.3pt]
6 & 0 & 0 & 1 & $\cdot$ & $\cdot$ & $\cdot$\\
\rowcolor{PrimeGray}\textbf{7} & 1 & 1 & 2 & 0 & $\cdot$ & $\cdot$\\
8 & 0 & 2 & 3 & 1 & $\cdot$ & $\cdot$\\
9 & 1 & 0 & 4 & 2 & $\cdot$ & $\cdot$\\
10 & 0 & 1 & 0 & 3 & $\cdot$ & $\cdot$\\
\addlinespace[1.3pt]
\rowcolor{PrimeGray}\textbf{11} & 1 & 2 & 1 & 4 & 0 & $\cdot$\\
12 & 0 & 0 & 2 & 5 & 1 & $\cdot$\\
\rowcolor{PrimeGray}\textbf{13} & 1 & 1 & 3 & 6 & 2 & 0\\
14 & 0 & 2 & 4 & 0 & 3 & 1\\
15 & 1 & 0 & 0 & 1 & 4 & 2\\
\bottomrule
\end{tabularx}
\end{minipage}\hfill
\begin{minipage}[t]{0.485\textwidth}
\centering
\textbf{Valuation trajectory $\V(n)$}\\[1mm]
\scriptsize
\setlength{\tabcolsep}{2.5pt}
\renewcommand{\arraystretch}{1.10}
\begin{tabularx}{\linewidth}{@{}c*{6}{>{\centering\arraybackslash}X}@{}}
\toprule
$n$ & $2$ & $3$ & $5$ & $7$ & $11$ & $13$\\
\midrule
1 & $\cdot$ & $\cdot$ & $\cdot$ & $\cdot$ & $\cdot$ & $\cdot$\\
\rowcolor{PrimeGray}\textbf{2} & 1 & $\cdot$ & $\cdot$ & $\cdot$ & $\cdot$ & $\cdot$\\
\rowcolor{PrimeGray}\textbf{3} & 0 & 1 & $\cdot$ & $\cdot$ & $\cdot$ & $\cdot$\\
4 & 2 & 0 & $\cdot$ & $\cdot$ & $\cdot$ & $\cdot$\\
\rowcolor{PrimeGray}\textbf{5} & 0 & 0 & 1 & $\cdot$ & $\cdot$ & $\cdot$\\
\addlinespace[1.3pt]
6 & 1 & 1 & 0 & $\cdot$ & $\cdot$ & $\cdot$\\
\rowcolor{PrimeGray}\textbf{7} & 0 & 0 & 0 & 1 & $\cdot$ & $\cdot$\\
8 & 3 & 0 & 0 & 0 & $\cdot$ & $\cdot$\\
9 & 0 & 2 & 0 & 0 & $\cdot$ & $\cdot$\\
10 & 1 & 0 & 1 & 0 & $\cdot$ & $\cdot$\\
\addlinespace[1.3pt]
\rowcolor{PrimeGray}\textbf{11} & 0 & 0 & 0 & 0 & 1 & $\cdot$\\
12 & 2 & 1 & 0 & 0 & 0 & $\cdot$\\
\rowcolor{PrimeGray}\textbf{13} & 0 & 0 & 0 & 0 & 0 & 1\\
14 & 1 & 0 & 0 & 1 & 0 & 0\\
15 & 0 & 1 & 1 & 0 & 0 & 0\\
\bottomrule
\end{tabularx}
\end{minipage}
\caption{The first fifteen positive integers as a recursively generated trajectory. Left: the finite cyclic seconds readings. Right: the derived nonnegative valuation readings. A newly discovered prime appears with clock state $(R_p,M_p)=(0,1)$ and valuation $V_p=1$. Before installation its valuation coordinate is dynamically absent; in the completed static representation it is zero.}
\label{fig:first-fifteen-trajectories}
\end{figure}

\begin{lemma}[Every composite has a smaller prime factor]
\label{lem:smaller-prime-factor}
If $n>1$ is composite, then some prime $p<n$ is a factor of $n$.
\end{lemma}

\begin{proof}
Among the factors $d>1$ of $n$, choose the least. If $d$ were composite, say $d=ab$ with $1<a<d$, then $a$ would also divide $n$, contradicting the minimality of $d$. Hence $d$ is prime; since $n$ is composite, $d<n$.
\end{proof}

\begin{proposition}[The recursion discovers exactly the primes]
\label{prop:prime-discovery}
At time $n>1$, the clockwork installs a new clock if and only if $n$ is prime.
\end{proposition}

\begin{proof}
Proceed by induction on $n$. Assume that before time $n$ the installed periods are exactly the primes below $n$. An installed $p$-clock rings at time $n$ exactly when $p\mid n$.

If $n$ is prime, no smaller installed prime divides it, so no seconds counter rings and $\mathcal C_n$ is installed. If $n$ is composite, Lemma~\ref{lem:smaller-prime-factor} gives a prime factor $p<n$. By induction $\mathcal C_p$ is already present, and its seconds counter rings at time $n$; hence no new clock is installed. The base case $n=2$ is immediate.
\end{proof}

The first growth events are therefore
\[
2,3,5,7,11,13,\ldots,
\]
but this list is output, not input. The concise Common Lisp program in Appendix~\ref{app:python} implements this online growth rule.

\section{Integers on the clockwork}

For positive integers, the seconds counters display congruence and divisibility alarms, the minutes counters record completed prime cycles, and the derived valuation readouts display prime multiplicity. This section develops the main integer results in that order: first the finite product dynamics, then factorization, and finally elementary arithmetic in the valuation coordinates.

\subsection{Grand cycles from Cartesian products}
\label{sec:grand-cycles}

After finitely many growth events, suppose the discovered periods are $p_1,\ldots,p_k$.  The state space of this finite stage is the Cartesian product
\begin{equation}
X_k=\mathcal C_{p_1}\times\cdots\times\mathcal C_{p_k},
\qquad
|X_k|=p_1p_2\cdots p_k.
\label{eq:cartesian-state-space}
\end{equation}
The common tick advances every component by one.  Before interpreting these positions as residue classes, we can determine the orbit directly from the cyclic clock geometry.  This is the ``grand cycle'' viewpoint developed informally in an earlier Mathematical Playground essay~\cite{EmmerichGrandCycles2024}.

\begin{lemma}[Nonzero rotation of a prime clock]
\label{lem:prime-rotation}
Let $p$ be prime.  Repeatedly rotate a $p$-clock by a fixed nonzero number of positions.  Starting from any position, the successive rotations visit all $p$ positions exactly once before returning.
\end{lemma}

\begin{proof}
A fixed rotation is a bijection of the $p$ positions.  Its repeated action partitions the positions into cycles.  Every cycle has the same length $\ell$.  If $\ell$ applications return one position to itself, the same total rotation returns every position, and minimality gives the same first-return time throughout.  Thus $p=c\ell$ for some positive integer $c$.  The rotation is nontrivial, so $\ell>1$; since $p$ is prime, $\ell=p$ and $c=1$.
\end{proof}

\begin{figure}[ht]
\centering
\begin{tikzpicture}[font=\small,>=Latex]
  \def\R{1.7}
  \begin{scope}
    \coordinate (c) at (0,0);
    \draw[ClockBlue,very thick] (c) circle[radius=\R];
    \foreach \k in {0,...,6}{
      \coordinate (v\k) at ({90-360/7*\k}:\R);
      \fill[ClockBlue] (v\k) circle[radius=1.6pt];
    }
    \node[above=3pt] at (v0) {$0$};
    \node[right=6pt] at (v1) {$1$};
    \node[right=6pt] at (v2) {$2$};
    \node[below=5pt] at (v3) {$3$};
    \node[below=5pt] at (v4) {$4$};
    \node[left=6pt] at (v5) {$5$};
    \node[left=6pt] at (v6) {$6$};
    \draw[-{Latex[length=2.2mm]},ClockGreen,thick] (v0) -- (v2);
    \draw[-{Latex[length=2.2mm]},ClockGreen,thick] (v2) -- (v4);
    \draw[-{Latex[length=2.2mm]},ClockGreen,thick] (v4) -- (v6);
    \draw[-{Latex[length=2.2mm]},ClockGreen,thick] (v6) -- (v1);
    \draw[-{Latex[length=2.2mm]},ClockGreen,thick] (v1) -- (v3);
    \draw[-{Latex[length=2.2mm]},ClockGreen,thick] (v3) -- (v5);
    \draw[-{Latex[length=2.2mm]},ClockGreen,thick] (v5) -- (v0);
  \end{scope}
\end{tikzpicture}
\caption{The local cyclic fact used both for the grand-cycle construction and, later, for unique factorization. On the $7$-clock a rotation by two positions gives the single orbit $0\to2\to4\to6\to1\to3\to5\to0$.}
\label{fig:clock-cycle}
\end{figure}

\begin{lemma}[Zero-product lemma on a prime clock]
\label{lem:clock-zero-product}
Let $p$ be prime and let $a,b$ be positive integers.  If the $p$-clock rings after $ab$ unit ticks, then it rings after $a$ ticks or after $b$ ticks.
\end{lemma}

\begin{proof}
Assume that the clock does not ring after $a$ ticks.  The position reached after those $a$ ticks is therefore a nonzero rotation.  Repeating this rotation $b$ times reaches the same position as $ab$ unit ticks.  By Lemma~\ref{lem:prime-rotation}, a nonzero rotation of the $p$-clock returns to $0$ only after a complete block of $p$ applications.  Hence, if the clock is at $0$ after $b$ applications, the ordinary $p$-clock also rings after $b$ unit ticks.
\end{proof}

\begin{definition}[Grand cycle]
For a finite stage of the recursively grown clockwork, the \emph{grand cycle} is the ordered orbit obtained by repeatedly applying the common unit tick to the Cartesian-product state~\eqref{eq:cartesian-state-space}.
\end{definition}

\begin{theorem}[Grand Cycle]
\label{prop:grand-cycle}
For discovered prime periods $p_1,\ldots,p_k$, the grand cycle has length
\begin{equation}
P_k=p_1p_2\cdots p_k
\label{eq:grand-cycle-length}
\end{equation}
and visits every element of $X_k$ exactly once before returning.
\end{theorem}

\begin{proof}
We argue inductively over growth events.  One prime clock is itself a single cycle, so the assertion holds for $k=1$.  Suppose the first $k$ clocks form one cycle of length $P_k$ and exhaust $X_k$.  When the next prime $p_{k+1}$ is discovered, attach its clock and observe the enlarged system in blocks of $P_k$ ticks.  After every such block all old coordinates have returned to their starting positions.  The new clock, however, has undergone the fixed rotation consisting of $P_k$ unit steps.

This rotation is nonzero; otherwise the $p_{k+1}$-clock would ring after $P_k=p_1\cdots p_k$ ticks.  Repeated application of Lemma~\ref{lem:clock-zero-product} would then force it to ring after some earlier prime $p_j<p_{k+1}$, which is impossible because a $p_{k+1}$-clock cannot complete a positive cycle in fewer than $p_{k+1}$ unit ticks.  Lemma~\ref{lem:prime-rotation} therefore shows that the successive $P_k$-tick blocks visit all $p_{k+1}$ positions of the new clock before returning.  The enlarged first-return time is consequently $P_kp_{k+1}$.

The Cartesian product contains exactly $|X_k|p_{k+1}=P_kp_{k+1}$ states.  No state can repeat before the first return.  If two intermediate states agreed, reversibility of the common tick would give an earlier return to the initial state.  Hence the orbit contains exactly as many distinct states as the whole Cartesian product and therefore visits every state once.
\end{proof}

This proof uses no modular notation and no Chinese remainder theorem.  It separates two elementary ingredients that are easy to see in the clock picture: Cartesian-product counting determines how many joint configurations exist, while the prime-clock rotation lemma determines how the new coordinate is threaded through the completed orbit of the old ones.

\begin{corollary}[Counting nonzero configurations]
\label{cor:grand-cycle-survivors}
In one grand cycle of $p_1,\ldots,p_k$, exactly
\begin{equation}
\prod_{j=1}^k(p_j-1)
\label{eq:grand-cycle-survivors}
\end{equation}
states have no clock at its zero position.
\end{corollary}

\begin{proof}
For each $p_j$-clock there are $p_j-1$ choices other than zero.  The admissible joint configurations form the Cartesian product of these nonzero position sets, so their number is the product of their cardinalities.  Theorem~\ref{prop:grand-cycle} ensures that each such configuration occurs exactly once along the ordered orbit.
\end{proof}

\subsubsection{The modular and CRT interpretation}
Once the positions of each discovered $p$-clock are labelled $0,1,\ldots,p-1$, the state at integer time $t$ is
\begin{equation}
C_k(t)=(t\bmod p_1,\ldots,t\bmod p_k).
\label{eq:jointstate}
\end{equation}
The common tick is the familiar coordinatewise update
\begin{equation}
T_k(x_1,\ldots,x_k)=(x_1+1,\ldots,x_k+1),
\label{eq:update}
\end{equation}
with each coordinate read on its own finite clock.  Theorem~\ref{prop:grand-cycle} says, before any appeal to CRT, that this temporal orbit visits every possible residue tuple exactly once in $P_k$ steps.

The Chinese remainder theorem now identifies the same fact algebraically through the bijection
\begin{equation}
\Phi:\Z/P_k\Z\longrightarrow\prod_{j=1}^k\Z/p_j\Z,
\qquad
\Phi(t)=(t\bmod p_1,\ldots,t\bmod p_k).
\label{eq:crt-map}
\end{equation}
Thus the clockwork does not replace CRT.  Rather, the grand-cycle proof constructs the finite orbit by recursion and counting, while CRT recognizes the resulting state space as a familiar product of residue rings.

Figure~\ref{fig:autonomous-array} displayed the same autonomous array before this algebraic interpretation was introduced. The dynamical statement is that the common tick orders the Cartesian product into one orbit; the CRT statement is that the same product is canonically a residue system modulo $P_k$.

\begin{remark}[State space versus trajectory]
The Cartesian product is the set of possible joint states; the grand cycle is the ordered trajectory through that set.  This distinction is useful because it turns a static simultaneous-congruence picture into a reversible finite dynamical system and makes return times, zero alarms, and sieve-like observables properties of an explicit orbit.
\end{remark}

\subsection{The valuation readout and factorization depth}
\label{sec:factorized-generation}

The early trajectory in Figure~\ref{fig:first-fifteen-trajectories} displays the valuations derived from the two clock hands. For a fixed prime $p$,
\[
V_p(n)=\nu_p(n)=\max\{e\ge0:p^e\mid n\}
\]
is the divisibility depth. Equivalently,
\begin{equation}
\boxed{V_p(n)=\sum_{k\ge1}\mathbf 1_{\{p^k\mid n\}}.}
\label{eq:valuation-depth}
\end{equation}
At $n=72$, for example, $2,4,8$ divide $72$ while $16$ does not, so $V_2(72)=3$; similarly $V_3(72)=2$. Thus
\[
\V(72)=(3,2,0,0,\ldots).
\]

The following additive law for the valuation readouts can be proved before invoking unique factorization.

\begin{lemma}[Primewise additivity of the valuation readout]
\label{lem:valuation-additive}
For every prime $p$ and positive integers $a,b$,
\begin{equation}
V_p(ab)=V_p(a)+V_p(b).
\label{eq:valuation-additive}
\end{equation}
Hence $\V(ab)=\V(a)+\V(b)$ coordinatewise.
\end{lemma}

\begin{proof}
Put $r=V_p(a)$ and $s=V_p(b)$. Then $a=p^r a_0$ and $b=p^s b_0$, where neither $a_0$ nor $b_0$ makes the $p$-clock ring. Thus $p^{r+s}\mid ab$. If one further factor of $p$ divided $ab$, the $p$-clock would ring on $a_0b_0$. By Lemma~\ref{lem:clock-zero-product} it would then ring on $a_0$ or on $b_0$, contradicting the maximality of $r$ or $s$. Therefore $V_p(ab)=r+s$.
\end{proof}

\begin{remark}[Valuation as a derived readout]
The integer minutes counter $M_p$ is part of the autonomous two-hand clock, whereas $V_p$ is obtained from~\eqref{eq:valuation-recursion}. The tests $p\mid n,p^2\mid n,p^3\mid n,\ldots$ could themselves be organized as another clock array, but that extra realization is unnecessary here. A ring, the current minutes address, and the clock's own valuation history are sufficient to generate the exponent depth.
\end{remark}

The next subsection first proves that the valuation trajectory never revisits a row. Together with primewise additivity, this dynamical non-repetition reconstructs every integer from its valuation row and yields unique factorization.

\subsection{Non-repetition and unique factorization from prime clocks}
\label{sec:clock-fta}

The grand-cycle construction supplied the local zero-product property of a prime clock, and Lemma~\ref{lem:valuation-additive} converted that property into coordinatewise additivity of the valuation readouts. Additivity now interacts with the temporal order of the unit-step trajectory. If a valuation row repeated, any positive coordinate would identify a prime clock ringing at both times, and its minutes readings would point to an earlier repeated pair.

\begin{proposition}[Non-repetition of valuation rows]
\label{prop:valuation-nonrepetition}
For positive integers $n$ and $m$,
\begin{equation}
\V(n)=\V(m)\quad\Longrightarrow\quad n=m.
\label{eq:valuation-injective}
\end{equation}
Thus the valuation trajectory $\V(1),\V(2),\V(3),\ldots$ has no repeated row.
\end{proposition}

\begin{proof}
The row at time $1$ is $\V(1)=\mathbf0$. At every time $t>1$, either an existing clock rings, in which case~\eqref{eq:valuation-recursion} gives a positive valuation coordinate, or no existing clock rings, in which case Proposition~\ref{prop:prime-discovery} installs the new $t$-clock with $V_t(t)=1$. Hence the zero row never occurs again.

Suppose that a repeated row nevertheless exists, and choose $1<n<m$ with $m$ minimal and $\V(n)=\V(m)$. The common row is nonzero, so choose a prime $p$ with
\[
V_p(n)=V_p(m)>0.
\]
The $p$-clock rings at both times. By~\eqref{eq:clock-division}, its minutes readings
\[
a=M_p(n),\qquad b=M_p(m)
\]
satisfy $n=pa$, $m=pb$, and $1\le a<b<m$. Since $\V(p)=\mathbf e_p$, Lemma~\ref{lem:valuation-additive} gives
\[
\V(n)=\mathbf e_p+\V(a),
\qquad
\V(m)=\mathbf e_p+\V(b).
\]
Subtracting the same coordinate unit $\mathbf e_p$ from the equal rows yields $\V(a)=\V(b)$. This is an earlier repeated pair because $a<b<m$, contradicting the minimal choice of $m$. Therefore no valuation row repeats.
\end{proof}

\begin{theorem}[Fundamental theorem of arithmetic, clock form]
\label{prop:clock-fta}
Every integer $n>1$ is a product of primes, and that product is unique up to the order of its factors.
\end{theorem}

\begin{proof}
For any finite-support nonnegative vector $\mathbf v=(v_p)_p$, form the integer
\[
F(\mathbf v):=\prod_p p^{v_p}.
\]
For each prime $p$, $\V(p)=\mathbf e_p$. Repeated application of Lemma~\ref{lem:valuation-additive} therefore gives
\begin{equation}
\V(F(\mathbf v))=\sum_p v_p\mathbf e_p=\mathbf v.
\label{eq:valuation-synthesis}
\end{equation}

For a fixed positive integer $n$, the row $\V(n)$ has finite support, since $V_p(n)=0$ whenever $p>n$. Taking $\mathbf v=\V(n)$ in~\eqref{eq:valuation-synthesis} gives
\[
\V(F(\V(n)))=\V(n).
\]
Both $F(\V(n))$ and $n$ are integer times on the same trajectory, so Proposition~\ref{prop:valuation-nonrepetition} forces
\[
n=F(\V(n))=\prod_p p^{V_p(n)}.
\]
This proves existence of a prime factorization directly from the valuation row. If also $n=\prod_p p^{a_p}$ for a finite-support family of nonnegative integers $(a_p)_p$, then additivity gives $\V(n)=(a_p)_p$. Hence $a_p=V_p(n)$ for every prime $p$, proving uniqueness up to the order of the factors.
\end{proof}

The local mathematical content remains classical: Lemma~\ref{lem:clock-zero-product} is the clock form of Euclid's prime-product property. The organization of the proof is intrinsic to the trajectory. A hypothetical repeated valuation row descends, through a ringing clock's minutes address, to an earlier repeated row. Once non-repetition is known, synthesizing a prime product with the same row reconstructs the original integer time. No prior factorization of that integer is assumed.

\subsubsection{The valuation readouts become intrinsic multiplicative coordinates}

The proof above identifies the exponent of each prime in the unique factorization of $n$ with the derived readout $V_p(n)$. Hence
\begin{equation}
\boxed{V_p(n)=\nu_p(n)=\sum_{k\ge1}\mathbf{1}_{\{p^k\mid n\}},
\qquad
n=\prod_p p^{V_p(n)}.}
\label{eq:valuation-clock-tower}
\end{equation}
The valuation vector is therefore not merely an annotation: the trajectory and the synthesis map $F$ give the bijection
\[
\N_{>0}\longleftrightarrow
\bigl\{(v_p)_p:v_p\in\Z_{\ge0}\text{ and }v_p=0\text{ for all but finitely many }p\bigr\}.
\]
The two clock hands are not needed to reconstruct the integer from this static representation; their role is dynamical and becomes important again when the integer is advanced additively.

\subsubsection{Elementary arithmetic in prime coordinates}
\label{sec:elementary}

The following results are classical.  Their interest here is how little remains to be proved once the exponent coordinates and residue dynamics are visible.

\subsubsection{Divisibility, gcd, and lcm}

For finite-support valuation vectors define the \emph{weak componentwise order}
\[
\mathbf u\preceq\mathbf v
\quad\Longleftrightarrow\quad
u_p\le v_p\quad\text{for every prime }p.
\]
If $\V(a)=(a_p)_p$ and $\V(b)=(b_p)_p$, unique factorization gives
\begin{equation}
\boxed{a\mid b\iff \V(a)\preceq\V(b)
       \iff a_p\le b_p\ \text{for every }p.}
\label{eq:divisibility}
\end{equation}
Thus divisibility is simply the weak product order on the valuation coordinates. Moreover,
\begin{align}
\V(\gcd(a,b))&=(\min(a_p,b_p))_p,\label{eq:gcdmin}\\
\V(\lcm(a,b))&=(\max(a_p,b_p))_p.\label{eq:lcmmax}
\end{align}
The identity $\gcd(a,b)\lcm(a,b)=ab$ then reduces in each prime coordinate to
\[
\min(r,s)+\max(r,s)=r+s.
\]

\begin{figure}[ht]
\centering
\begin{tikzpicture}[x=1.2cm,y=1.2cm,font=\small]
  \draw[-{Latex[length=2.5mm]},ClockBlue!75] (-0.15,0)--(4.25,0) node[right] {$\nu_2$};
  \draw[-{Latex[length=2.5mm]},ClockBlue!75] (0,-0.15)--(0,4.25) node[above] {$\nu_3$};
  \foreach \x in {0,...,4}{\draw[ClockBlue!18] (\x,0)--(\x,4); \node[below=1mm] at (\x,0) {\x};}
  \foreach \y in {1,...,4}{\draw[ClockBlue!18] (0,\y)--(4,\y); \node[left=1mm] at (0,\y) {\y};}
  \fill[ClockBlue] (3,2) circle[radius=2.1mm];
  \node[anchor=west] at (3.12,2.16) {$72=2^3 3^2$};
  \fill[ClockGreen] (2,3) circle[radius=2.1mm];
  \node[anchor=east] at (1.88,3.22) {$108=2^2 3^3$};
  \fill[ClockGold] (2,2) circle[radius=2.1mm];
  \node[anchor=north east,align=right] at (1.83,1.88) {$\gcd=36$\\coordinatewise min};
  \fill[ClockRed] (3,3) circle[radius=2.1mm];
  \node[anchor=south west,align=left] at (3.12,3.12) {$\lcm=216$\\coordinatewise max};
  \draw[dashed,ClockGold!80] (2,2)--(3,2)--(3,3);
  \draw[dashed,ClockGreen!70] (2,2)--(2,3)--(3,3);
\end{tikzpicture}
\caption{In a two-prime slice, gcd and lcm are simply the lower-left and upper-right corners determined by the two exponent vectors.}
\label{fig:gcdlcm}
\end{figure}
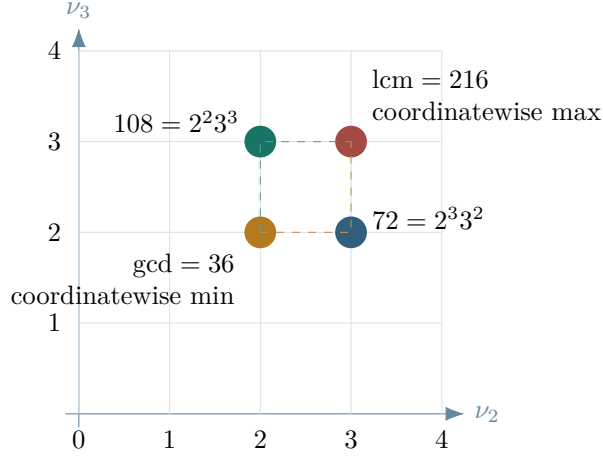

\subsubsection{Perfect powers and elementary root tests}

If
\[
n=\prod_p p^{e_p},
\]
then
\begin{equation}
\boxed{n\text{ is a }k\text{-th power}\iff k\mid e_p\text{ for every }p.}
\label{eq:kthpower}
\end{equation}
Thus perfect squares occupy the even sublattice $2\Z_{\ge0}^{(\Pset)}$, cubes occupy $3\Z_{\ge0}^{(\Pset)}$, and so forth. The largest perfect-power exponent of $n>1$ is
\[
\gcd\{e_p:e_p\ne0\}.
\]

The familiar irrationality of $\sqrt2$ becomes a one-coordinate obstruction. If $\sqrt2=a/b\in\Q_{>0}$, then
\[
2\V(a)=\V(2)+2\V(b).
\]
Looking only at the $2$-coordinate gives an even integer on the left and an odd integer on the right, impossible. Equivalently, the vector $(1,0,\ldots)$ of $2$ is not in the even lattice.

\subsubsection{Euler's totient as a grand-cycle observable}

Let
\[
n=\prod_{j=1}^{r}p_j^{a_j},
\qquad
P=\prod_{j=1}^{r}p_j.
\]
An integer time is coprime to $n$ exactly when none of the clocks $p_1,\ldots,p_r$ rings.  Corollary~\ref{cor:grand-cycle-survivors} counts these states without using modular arithmetic. Among the $P$ positions of one grand cycle, exactly
\[
\prod_{j=1}^{r}(p_j-1)
\]
have no zero coordinate.  Since $P$ divides $n$, the same prime-clock pattern occurs exactly $n/P$ times during the first $n$ integer ticks.  Hence
\begin{equation}
\boxed{\varphi(n)=\frac{n}{P}\prod_{j=1}^{r}(p_j-1)
      =n\prod_{p\mid n}\left(1-\frac1p\right)
      =\prod_{p^a\parallel n}p^{a-1}(p-1).}
\label{eq:totient}
\end{equation}
For a prime power this reduces at once to $\varphi(p^a)=p^{a-1}(p-1)$.  The usual CRT interpretation remains available, but the count itself is already visible from the Cartesian-product grand cycle.

\begin{figure}[ht]
\centering
\begin{tikzpicture}[font=\scriptsize]
  \def\w{0.72}
  \def\h{0.62}
  \foreach \n in {1,...,30}{
    \pgfmathtruncatemacro{\row}{(\n-1)/15}
    \pgfmathtruncatemacro{\col}{mod(\n-1,15)}
    \pgfmathsetmacro{\xx}{\col*\w}
    \pgfmathsetmacro{\yy}{-\row*\h}
    \pgfmathtruncatemacro{\rTwo}{mod(\n,2)}
    \pgfmathtruncatemacro{\rThree}{mod(\n,3)}
    \pgfmathtruncatemacro{\rFive}{mod(\n,5)}
    \ifnum\rTwo>0
      \ifnum\rThree>0
        \ifnum\rFive>0
          \fill[SoftGreen] (\xx,\yy) rectangle ++(\w-0.04,\h-0.04);
          \draw[ClockGreen,thick] (\xx,\yy) rectangle ++(\w-0.04,\h-0.04);
        \else
          \fill[PaleGray] (\xx,\yy) rectangle ++(\w-0.04,\h-0.04);
          \draw[ClockBlue!25] (\xx,\yy) rectangle ++(\w-0.04,\h-0.04);
        \fi
      \else
        \fill[PaleGray] (\xx,\yy) rectangle ++(\w-0.04,\h-0.04);
        \draw[ClockBlue!25] (\xx,\yy) rectangle ++(\w-0.04,\h-0.04);
      \fi
    \else
      \fill[PaleGray] (\xx,\yy) rectangle ++(\w-0.04,\h-0.04);
      \draw[ClockBlue!25] (\xx,\yy) rectangle ++(\w-0.04,\h-0.04);
    \fi
    \node at ({\xx+(\w-0.04)/2},{\yy+(\h-0.04)/2}) {\n};
  }
\end{tikzpicture}
\caption{The totient as a survival count on the ordered clockwork orbit. For the $2$-, $3$-, and $5$-clocks, the eight green states avoid every zero alarm, so $\varphi(30)=8$.}
\label{fig:totient}
\end{figure}

\subsubsection{B\'ezout through valuation coordinates and clocks}

The exponent layer first removes the common multiplicative content. Put
\[
d=\gcd(a,b),\qquad a=da',\qquad b=db'.
\]
By~\eqref{eq:gcdmin}, $a'$ and $b'$ have no positive prime coordinate in common, hence $\gcd(a',b')=1$.

Now consider multiplication by $a'$ on the residue clock modulo $b'$. If
\[
ra'\equiv sa'\pmod{b'},
\]
then $b'\mid(r-s)a'$. Prime-exponent comparison and the absence of common prime factors imply $b'\mid r-s$. Thus multiplication by $a'$ is injective on the finite set $\Z/b'\Z$, hence bijective. Some residue $x$ therefore satisfies
\[
xa'\equiv1\pmod{b'}.
\]
So $xa'-1=kb'$ for an integer $k$, and
\begin{equation}
xa'+(-k)b'=1.
\end{equation}
Multiplying by $d$ yields
\begin{equation}
\boxed{xa+(-k)b=\gcd(a,b).}
\label{eq:bezout}
\end{equation}

\begin{remark}[Why the two layers cooperate]
The valuation coordinates perform the multiplicative simplification by removing the coordinatewise minimum. The residue clock then supplies the inverse motion. Once the common prime factors are gone, multiplication by $a'$ permutes the $b'$ residue states, so one of the states is $1$. The proof is elementary and classical; the framework makes the division of labor between factorization and congruence unusually explicit.
\end{remark}

\begin{remark}[Local non-closure and global reconstruction]
\label{rem:global-reconstruction}
For positive integers, $\V(n)$ represents $n$ uniquely. Under multiplication the valuation projection closes on itself,
\[
\V(ab)=\V(a)+\V(b),
\]
so an observer who watches only the valuation coordinates does not need to see the clock hands. In this restricted multiplicative view the seconds and minutes readings are \emph{hidden variables}. This should not be confused with autonomy of the prime clocks themselves. Under $+1$, each isolated clock updates its two hands directly by~\eqref{eq:two-hand-update}; its valuation readout is then obtained from those hands and its own earlier valuation history. The valuation coordinate alone is not a closed local state description for $n\mapsto n+1$ or for general additive evolution.

Taken together, however, the valuation readings contain sufficient global information. By Theorem~\ref{prop:clock-fta}, the finite-support vector $\V(n)$ determines $n$ uniquely. It therefore determines $n+1$, every current pair $(R_p(n),M_p(n))$, and the complete successor state, including all successor clock and valuation readings and whether a new clock is installed. Thus the hidden clock hands are reconstructible rather than lost.

This is global determinacy, not a coordinatewise local recurrence: reconstruction passes through the represented integer and may require factorizing its successor. Maintaining the two-hand states explicitly instead makes $n\mapsto n+1$ local on every clock and, more generally,
\[
R_p(a+b)=R_p(a)+R_p(b)\pmod p,
\]
whereas no analogous coordinatewise update exists for $\V(a+b)$. Thus the clock hands provide the directly updated dynamical variables, while the valuation vector is the sufficient static multiplicative representation.
\end{remark}

\section{From integers to rational numbers}

Passing from positive integers to positive rationals extends the derived valuation coordinates rather than the elapsed-cycle minutes counters. Valuations become signed integers, while a seconds residue remains finite and is defined only when the denominator is invertible modulo the relevant prime.

\subsection{Signed valuation vectors from integer pairs}

The extension to positive rationals can be carried out entirely in the clockwork's additive prime-vector logic, without taking division as a primitive operation. This is the clockwork reformulation of the prime-vector argument developed in~\cite{EmmerichPrimeVectors2025}. Let $(a,b)$ and $(c,d)$, with positive integer entries, be equivalent when
\[
ad=bc,
\]
and denote the equivalence class of $(a,b)$ by $[a,b]$. Define
\begin{equation}
\boxed{\V([a,b])=\V(a)-\V(b).}
\label{eq:rational-pair-vector}
\end{equation}
If $(a,b)\sim(c,d)$, then $ad=bc$. Primewise additivity on integers gives
\[
\V(a)+\V(d)=\V(ad)=\V(bc)=\V(b)+\V(c),
\]
so $\V([a,b])=\V([c,d])$. Thus the signed valuation vector is well defined.

The decisive uniqueness step is additive. Suppose two pair classes have the same signed valuation vector,
\[
\V([a,b])=\V([c,d]).
\]
Then
\[
\V(a)-\V(b)=\V(c)-\V(d),
\]
and adding the same valuation vectors to both sides yields
\[
\V(a)+\V(d)=\V(c)+\V(b).
\]
Integer additivity identifies these as the valuation vectors of $ad$ and $cb$. The already proved uniqueness of the integer valuation representation therefore gives $ad=cb$, which is precisely the equivalence relation $[a,b]=[c,d]$. Thus rational uniqueness is obtained by addition and equality in prime-vector space, without quotient or product cancellation.

Surjectivity is equally coordinatewise. Given any finite-support integer vector $\mathbf z=(z_p)_p$, split it into its nonnegative parts,
\[
\mathbf z=\mathbf z^{+}-\mathbf z^{-},
\qquad z_p^{+}=\max(z_p,0),\quad z_p^{-}=\max(-z_p,0).
\]
By the positive-integer representation theorem there are unique positive integers $A$ and $B$ with $\V(A)=\mathbf z^{+}$ and $\V(B)=\mathbf z^{-}$. Then $\V([A,B])=\mathbf z$. Hence
\begin{equation}
\boxed{\Q_{>0}^{\times}\cong\bigoplus_{p\ \mathrm{prime}}\Z,}
\label{eq:rational-clock-bijection}
\end{equation}
and multiplication of positive rationals corresponds exactly to addition of signed valuation vectors.

Thus positive rationals, like positive integers, are uniquely represented by the valuation coordinates. If one observes only these coordinates during multiplication, the clock hands are hidden variables of the multiplicative projection. They remain essential to the direct unit-step clockwork on integers. Because the signed vector determines the rational number, the corresponding finite seconds residues remain reconstructible wherever they are defined.

\begin{remark}[What is, and is not, being claimed]
The theorem is classical. The prime-clock zero-product lemma carries the same mathematical content as Euclid's prime-product property, while the valuation readouts convert that local fact into additive prime coordinates. Alternative proofs of the fundamental theorem of arithmetic have a substantial history~\cite{Dawson2015}. No priority claim is made for this formulation. Its role here is expository: the uniqueness step is carried out inside the prime clockwork by equality and subtraction of valuation vectors, rather than by reverting to ordinary product cancellation.
\end{remark}

\subsection{Backward and rational-step motion of the seconds counters}

Autonomy makes every finite seconds counter reversible. The update $x\mapsto x+1$ on $\mathcal C_p$ has inverse $x\mapsto x-1$.

The seconds counters can also be parameterized by rational arithmetic steps. Let
\[
q=\frac ab\in\Q,\qquad \gcd(a,b)=1.
\]
If $p\nmid b$, then $b$ is invertible modulo $p$, and one step of size $q$ on the $p$-seconds counter is the permutation
\begin{equation}
T_{p,q}(x)=x+a\,b^{-1}\pmod p.
\label{eq:rational-step}
\end{equation}
Thus an autonomous prime clock can run forward, backward, or through an admissible rational arithmetic step while its seconds counter remains in the finite state space $\mathbb Z/p\mathbb Z$.

The same construction assigns residue states to many positive rationals. For
\[
x=\frac ab>0,\qquad \gcd(a,b)=1,
\]
with $p\nmid b$, define
\begin{equation}
R_p(x)=a\,b^{-1}\pmod p.
\label{eq:rational-residue}
\end{equation}
For example, on the $7$-clock,
\[
\frac53\equiv5\cdot3^{-1}\equiv4\pmod7.
\]
The restriction $p\nmid b$ is essential. The rational number $1/2$ has a residue on every odd prime clock but not on the $2$-clock. Whenever the relevant residues exist, the seconds state also carries addition locally:
\[
R_p(x+y)=R_p(x)+R_p(y)\pmod p.
\]
Thus the same hidden seconds variables that support successor dynamics on integers continue to carry additive information for rationals on every admissible prime clock.

The valuation coordinates record the same obstruction from the multiplicative side. If $x=a/b$ is reduced, then
\[
\vp(x)=\vp(a)-\vp(b),
\]
and therefore
\begin{equation}
\vp(x)<0\quad\Longleftrightarrow\quad p\mid b.
\label{eq:negative-missing}
\end{equation}
This is exactly the condition under which the finite $p$-residue is unavailable.

\begin{remark}[Two meanings of a rational movement]
Rational phase motion concerns the finite seconds counter and is a modular permutation available when the denominator is prime to $p$. Rational valuation levels, introduced later, are different: they interpolate the multiplicative exponent and allow values such as $1/2$ or $2/3$. They are not elapsed-cycle minutes readings.
\end{remark}

\subsection{Valuation vectors for rational numbers}

For positive integers the valuation readout attached to prime $p$ is
\begin{equation}
V_p(n)=\nu_p(n)=\max\{k\ge0:p^k\mid n\}
       =\sum_{k\ge1}\mathbf{1}_{\{p^k\mid n\}}.
\end{equation}
Collect these readouts into
\[
\V(n)=(\nu_2(n),\nu_3(n),\nu_5(n),\ldots).
\]
For $n\in\N_{>0}$ this is the vector of nonnegative integer valuation readings of the prime clockwork and, by unique factorization, a unique representation of $n$.

For positive rationals, define $\nu_p(a/b)=\nu_p(a)-\nu_p(b)$. The valuation coordinates then take signed integer values, and unique factorization gives
\begin{equation}
\boxed{x=\prod_p p^{\nu_p(x)}}
\label{eq:reconstruction}
\end{equation}
with a unique finite-support vector. In particular,
\begin{equation}
\V(xy)=\V(x)+\V(y),\qquad \V(x^{-1})=-\V(x).
\label{eq:linear}
\end{equation}
This is coordinatewise independence in multiplicative form: for each fixed prime $p$, the valuation changes independently of all other prime coordinates.

Thus
\[
\N_{>0}\cong\Z_{\ge0}^{(\Pset)},\qquad
\Q_{>0}^{\times}\cong\Z^{(\Pset)}.
\]
For instance,
\[
360=2^3\,3^2\,5\longleftrightarrow(3,2,1,0,\ldots),
\]
while
\[
\frac{12}{35}=2^2\,3\,5^{-1}7^{-1}\longleftrightarrow(2,1,-1,-1,0,\ldots).
\]

\begin{remark}[Clock hands versus valuation coordinates]
The minutes counter $M_p$ belongs to the unit-step clock and counts completed $p$-cycles; the valuation $V_p$ is its history-dependent multiplicative readout. Under multiplication the valuation coordinates are independent, while under $+1$ the two clock hands update directly by carry. Keeping these roles separate prevents elapsed cycles from being confused with prime-exponent depth.
\end{remark}

\section{Algebraic numbers}

The valuation coordinates permit one further elementary extension. Allowing rational rather than integer exponent levels reaches a nontrivial class of positive algebraic irrationalities. This is an extension of the multiplicative coordinates, not a continuous interpolation of either clock hand.

\subsection{Rational valuation levels}

For positive integers the valuation coordinates take nonnegative integer values, and for positive rationals they take signed integer values. A further interpolation is obtained by allowing rational rather than integer levels in these coordinates. These rational valuation levels are algebraically rational prime exponents; they are not states of the unit-step clock hands. This already reaches certain positive irrational algebraic numbers, although not all of them. Define
\begin{equation}
\G=
\left\{\prod_p p^{q_p}:q_p\in\Q,\ q_p=0\text{ for all but finitely many }p\right\}
\subset\mathbb R_{>0},
\label{eq:G}
\end{equation}
where the positive real value of each rational power is used. For $x\in\G$, write
\[
\VQ(x)=(q_2,q_3,q_5,\ldots).
\]
Then
\begin{equation}
\VQ(xy)=\VQ(x)+\VQ(y),
\qquad
\VQ(x^r)=r\VQ(x)\quad(r\in\Q).
\label{eq:fractional-linear}
\end{equation}
Powers scale every valuation coordinate, while roots divide every coordinate.

\begin{proposition}[Exact scope of the rational-exponent representation]
\label{prop:Gcharacterization}
For $x>0$,
\begin{equation}
\boxed{x\in\G\iff x^N\in\Q_{>0}\text{ for some integer }N\ge1.}
\label{eq:Gcharacterization}
\end{equation}
Moreover, the rational exponent vector of $x$ is unique.
\end{proposition}

\begin{proof}
If $x=\prod p^{q_p}$, choose $N$ divisible by the finitely many denominators of the $q_p$. Then $x^N$ has integer prime exponents and is a positive rational. Conversely, if $x^N=r\in\Q_{>0}$, write $r=\prod p^{e_p}$ and take the positive $N$th root to obtain $x=\prod p^{e_p/N}$. For uniqueness, if $\prod p^{q_p}=1$, clear all denominators with a common $N$; unique prime factorization of the resulting rational identity forces every $Nq_p$, hence every $q_p$, to vanish.
\end{proof}

Thus
\[
\sqrt2\longleftrightarrow(1/2,0,0,\ldots),
\qquad
\sqrt[3]{12}\longleftrightarrow(2/3,1/3,0,\ldots),
\]
and
\[
\sqrt{\frac35}\longleftrightarrow(0,1/2,-1/2,0,\ldots).
\]
Every positive rational root of a positive rational is represented in this way. Thus the rational-step extension includes a nontrivial class of irrational algebraic numbers, while Section~\ref{sec:natural-boundary} shows that it does not include all algebraic numbers.

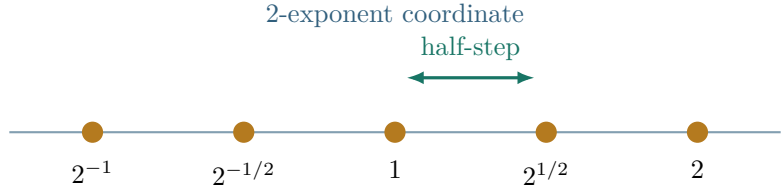
\begin{figure}[ht]
\centering
\begin{tikzpicture}[font=\small]
  \draw[ClockBlue!60,thick] (-5.1,0)--(5.1,0);
  \foreach \x/\lab in {-4/{$2^{-1}$},-2/{$2^{-1/2}$},0/{$1$},2/{$2^{1/2}$},4/{$2$}}{
    \draw[ClockBlue!60] (\x,-0.12)--(\x,0.12);
    \fill[ClockGold] (\x,0) circle[radius=1.4mm];
    \node[below=2.5mm] at (\x,0) {\lab};
  }
  \draw[{Latex[length=2.3mm]}-{Latex[length=2.3mm]},ClockGreen,very thick] (0.15,0.72)--(1.85,0.72);
  \node[ClockGreen,above=1mm] at (1.0,0.72) {half-step};
  \node[ClockBlue] at (0,1.55) {$2$-exponent coordinate};
\end{tikzpicture}
\caption{Fractional prime exponents refine the exponent coordinate itself. Adjacent half-levels differ by multiplication by $2^{1/2}$, while the finite residue clocks are not subdivided.}
\label{fig:fractional}
\end{figure}

\subsection{A natural boundary}
\label{sec:natural-boundary}

The class $\G$ is multiplicatively rich, but it is not closed under addition. This is not a technical inconvenience; it is the simplest precise boundary of the prime-exponent representation.

\begin{proposition}
Neither $1+\sqrt2$ nor the golden ratio
\[
\varphi=\frac{1+\sqrt5}{2}
\]
belongs to $\G$.
\end{proposition}

\begin{proof}
Let $\alpha=1+\sqrt2$. If $\alpha\in\G$, Proposition~\ref{prop:Gcharacterization} gives $\alpha^N\in\Q$ for some $N\ge1$. Applying the conjugation $\sqrt2\mapsto-\sqrt2$ would then give
\[
(1+\sqrt2)^N=(1-\sqrt2)^N.
\]
But $1+\sqrt2>1$ while $|1-\sqrt2|<1$, so the two sides have different absolute values. This is impossible.

For $\varphi$, its conjugate is
\[
\psi=\frac{1-\sqrt5}{2}=-\varphi^{-1}.
\]
If $\varphi^N\in\Q$, conjugation would force $\varphi^N=\psi^N$. Again $|\varphi|>1$ and $|\psi|<1$, a contradiction.
\end{proof}

\begin{remark}[Natural boundary]
This paper stops here deliberately. One can enlarge the ambient algebraic machinery using number fields, prime ideals, units, places, and related constructions, but that requires additional structure. In its elementary form, the second layer of the prime clockwork is the rational extension of the prime-indexed valuation coordinates $\nu_p$. It linearizes multiplication, division, powers, and roots of prime-radical monomials. Ordinary addition is not turned into coordinatewise motion. The failure of $1+\sqrt2$ and $\varphi$ to fit into one rational prime-exponent state is therefore an informative boundary of the representation.
\end{remark}

\section{What the clockwork reveals}

The value of the representation is easiest to assess by comparing the structures that become adjacent in it. The same prime-indexed organization contains recursion, finite-state dynamics, modular phase, multiplicative depth, and counting observables. Standard theorems remain standard; the difference lies in their order of appearance and in the transitions between them.

The prime clockwork can be summarized as a recursively grown array of autonomous two-hand clocks. Their maintained state is
\[
(R_p(n),M_p(n)),
\]
where seconds advance cyclically and each completed $p$-cycle carries one unit to minutes. Starting from the zero state, the common unit tick generates these trajectories, triggers new-clock growth, and produces at every finite stage one grand cycle through the Cartesian product of the seconds state spaces.

The multiplicative reading
\[
x\longmapsto V_p(x)=\nu_p(x)
\]
is the representational layer. At integer times it is generated by the same $p$-clock from a ring, the current minutes address, and local valuation history. For positive integers the finite-support nonnegative vector $\V(x)$ is unique; for positive rationals its signed extension is unique. Under multiplication each coordinate evolves independently by $V_p(xy)=V_p(x)+V_p(y)$. If only the valuation vector is watched, the two clock hands are hidden variables of this multiplicative projection. They remain reconstructible from the represented global state, but maintaining them makes local unit-step evolution immediate and avoids reconstructing and refactorizing the integer at every additive step.

Several consequences become concise in this language. Prime periods are generated growth events rather than an a priori parameter list, and the online construction uses successor, comparison, zero tests, carry, and lookup rather than explicit \texttt{mod} or \texttt{div}. Divisibility is the weak componentwise order on valuation vectors, division is coordinatewise subtraction, gcd and lcm are componentwise minimum and maximum, and perfect powers form sublattices. Euler's totient counts grand-cycle states that avoid the zero alarms of the primes dividing $n$, while B\'ezout combines coordinatewise removal of common valuation content with inverse motion in the residue layer. Finally, rational valuation levels represent exactly the positive reals for which some positive integer power is rational.

None of these statements is intended as a new theorem. The proposal is that their order and common prime-indexed organization may be mathematically and pedagogically useful. The same language passes from prime discovery to finite product dynamics, from product dynamics to modular arithmetic, and from prime exponents to a controlled rational extension. This may provide a modest bridge between elementary number theory, visual reasoning, and discrete dynamical systems.

\section{Further perspectives}

The prime clockwork should not be read as a claim that clocks, congruences, valuations, or the underlying number-theoretic theorems are new. The question is whether their particular synthesis produces a useful expository and proof language. The recursive emergence of prime periods, the proof of the grand cycle before invoking the Chinese remainder theorem, the paired seconds/minutes readings, and the controlled extension from integer to signed and rational valuation coordinates are the aspects on which that claim rests.

Several components of the present exposition were developed previously in the author's \emph{Mathematical Playground} essays.  Prime-vector notation for integers and positive rationals was introduced as a positional representation based on prime exponents~\cite{EmmerichPlaying2024,EmmerichPrimeVectors2025}; the latter essay emphasized the bijection between positive rationals and finite-support integer prime vectors and gave an earlier uniqueness argument.  The recursive addition of prime clocks, long recurring clock patterns, Cartesian-product counting, and a complex-plane visualization were explored informally in ``Grand Cycles of the Primes''~\cite{EmmerichGrandCycles2024}.  A later LISP note implemented online prime detection by autonomous counters using increment and comparison as the arithmetic operations~\cite{EmmerichLisp2025}.  The present manuscript reorganizes these exploratory notes around the self-growing discrete dynamical system. Proposition~\ref{prop:prime-discovery} makes prime emergence explicit, Theorem~\ref{prop:grand-cycle} proves the grand cycle before invoking CRT, equations~\eqref{eq:two-hand-update} and~\eqref{eq:valuation-recursion} separate the autonomous clock motion from its derived exponent readout, and Section~\ref{sec:clock-fta} derives unique factorization from non-repetition of the valuation trajectory.

Clock representations of congruences are standard expository devices. Wilson's introductory number-theory text includes a chapter explicitly organized around ``congruences, clocks, and calendars''~\cite{Wilson2020}. More specifically, Perucca's \emph{Chinese Remainder Clock} gives an analog multi-hand clock and develops the CRT in terms of rotations~\cite{Perucca2017}. The present clock layer is therefore not claimed as a new visualization of the CRT.

Cyclic sieve constructions also organize residue classes. Pritchard's work \emph{Explaining the wheel sieve} develops a mathematical framework for such a cyclic sieve~\cite{Pritchard1982}; Holt and Rudd study cycles of gaps generated through stages of Eratosthenes' sieve and explicitly discuss a dynamical system underlying their recursion~\cite{HoltRudd2014}. These works are closer to sieving and gap structure than to the paired residue/valuation representation used here.

Alternative proofs of the fundamental theorem of arithmetic are themselves a well-developed subject; Dawson surveys several and discusses why alternative proofs may be mathematically useful even when the theorem is classical~\cite{Dawson2015}. Section~\ref{sec:clock-fta} should be read in that spirit. Its local prime-clock product lemma is equivalent in substance to the classical prime-product property, but it is derived from the cycle structure of a prime-period rotation inside a Cartesian product of autonomous prime clocks. The local clock law yields additivity of the valuation readouts. A minimal-repetition argument then descends through a ringing clock's minutes address, cancelling one valuation-coordinate unit and producing an impossible earlier repetition. Reconstruction from the resulting injective trajectory gives unique factorization. No priority claim is made for this formulation; its role is to connect a classical theorem directly to the representation used throughout the paper.

The exponent identities
\[
\nu_p(xy)=\nu_p(x)+\nu_p(y),
\qquad
\nu_p(x/y)=\nu_p(x)-\nu_p(y),
\]
and the coordinate descriptions of divisibility, gcd, lcm, and perfect powers are direct consequences of unique factorization and the standard additive valuation; see, for example, classical number-theory texts such as Apostol~\cite{Apostol1976} and Hardy--Wright~\cite{HardyWright2008}.

The intended distinction of the prime clockwork is one of synthesis, emphasis, and order of presentation. The construction is introduced explicitly as a recursive discrete trajectory: the common unit step drives each clock's autonomous seconds motion and carry into completed-cycle minutes, the absence of a zero alarm creates a new prime period without prior information about the primes, and the resulting growth times are exactly the primes. Each finite seconds stage is then treated as a Cartesian-product dynamical system whose grand cycle is proved before CRT is invoked. At ring times, the minutes reading addresses the clock's earlier valuation history, generating the prime-exponent depth. The valuation rows are shown not to repeat by descent through these minutes addresses, and unique factorization follows by reconstructing the integer time from its row. Signed valuation coordinates turn multiplicative composition and inversion into forward and backward coordinate motion. In that multiplicative projection both clock hands are hidden, but retaining them makes local $+1$ evolution immediate. Rational valuation levels then show how far the prime-indexed representation reaches into algebraic irrationality before ordinary addition leaves the model.

There are several directions one could pursue further. Richer algebraic number fields would require prime ideals, units, and additional local data, and the ordered grand cycles invite comparison with more general finite dynamical systems. A particularly natural question is whether the prime clocks, and the prime clockwork assembled from them, admit a useful continuous counterpart. Appendix~\ref{app:complex} introduces a distinguished unit clock whose positive-real crossings mark the integer events and embeds each discovered prime seconds clock by a complex exponential. In this continuous picture a new prime is signalled when the unit clock reaches an integer event while none of the previously installed prime clocks lies on its positive real reference ray. This is the continuous event-detection counterpart of the discrete growth rule. It remains open here whether the embedding leads to useful continuous analogues of grand cycles, clockwork observables, or related dynamical structures. These extensions are deliberately not developed here. The present paper is intended to isolate the elementary core and its natural boundary.

The contribution is therefore best viewed as representational and organizational rather than as a collection of new number-theoretic theorems. Its intended use is in visual and intuitive number theory, in the search for alternative proofs and explanations, and as a possible entry point from elementary modular arithmetic into dynamical-systems language.

The author hopes that, rather than literally reinventing the wheel, the prime clockwork contributes a useful new perspective from which to view the modular structure of integers, rationals, and algebraic numbers, and to formulate aspects of the corresponding theory in a more intuitive way.

\section*{Acknowledgements}
The author thanks Mutlu \"Ozdemir for discussions on the prime clockwork and related ideas.

\section*{Tool and computational resource disclosure}
Following the Leiden Declaration~\cite{LeidenDeclaration2026}, the author discloses the use of OpenAI's ChatGPT for language editing and as an auxiliary check for possible gaps or inconsistencies in mathematical arguments. It was neither treated as an author nor relied upon as a mathematical authority. The author verified the arguments and references and assumes full responsibility for the manuscript and any remaining errors.

\paragraph{License.}
This manuscript is intended for release under the Creative Commons Attribution 4.0 International license (CC BY 4.0).

\appendix
\section{A compact executable form of the prime clockwork}
\label{app:python}

\begin{figure}[H]
\centering
\begin{minipage}{0.96\textwidth}
\textbf{Python implementation. Two-hand clocks and their valuation readouts.}
\begin{lstlisting}[style=clockpython]
"""Compact division-free prime clockwork."""

def advance(r, m, p):
    return (0, m + 1) if r + 1 == p else (r + 1, m)

def clockwork(N):
    P, seconds, minutes, history = [], {}, {}, {}
    rows = [(1, (), (), ())]
    for n in range(2, N + 1):
        for p in P:
            seconds[p], minutes[p] = advance(
                seconds[p], minutes[p], p)
        if all(seconds[p] != 0 for p in P):
            P.append(n)
            seconds[n], minutes[n], history[n] = 0, 1, {}
        V = {}
        for p in P:
            V[p] = (0 if seconds[p] else
                    1 + history[p].get(minutes[p], 0))
            history[p][n] = V[p]
        residues = tuple((p, seconds[p]) for p in P)
        cycles = tuple((p, minutes[p]) for p in P)
        factors = tuple((p, V[p]) for p in P if V[p])
        rows.append((n, residues, cycles, factors))
    return tuple(P), rows

if __name__ == "__main__": print(clockwork(15))
\end{lstlisting}
\vspace{1mm}
\textbf{Walkthrough.} Each loop is one unit step. Every installed clock advances its seconds hand and carries one unit to its minutes hand on returning to zero. If none rings, the current integer is installed as a new period in state $(0,1)$. At a ring, the minutes reading addresses that clock's own valuation history and~\eqref{eq:valuation-recursion} adds one; away from a ring the valuation is zero. The program contains neither a remainder nor a division operator. Its output generates the primes, all active seconds and minutes readings, and the nonzero valuation coordinates of the factorization.

\vspace{1mm}
\textbf{Common Lisp implementation. The minimal online prime-clock process.}
\begin{lstlisting}[style=clocklisp]
(let ((cs nil) (n 2))
  (labels ((composite-p ()
             (some (lambda (c) (zerop (cdr c))) cs))
           (tick ()
             (dolist (c cs)
               (incf (cdr c))
               (when (= (cdr c) (car c)) (setf (cdr c) 0)))))
    (loop
      (if (composite-p)
          (progn (tick) (incf n))
          (progn (print n) (push (cons n 0) cs)
                 (tick) (incf n))))))
\end{lstlisting}
\textbf{Walkthrough.} \texttt{cs} stores the discovered clocks. \texttt{composite-p} tests for a zero hand, \texttt{tick} advances and resets, and \texttt{push} installs the new clock when none rings~\cite{EmmerichLisp2025}.
\end{minipage}
\caption{Two executable views of the prime clockwork. The Common Lisp source isolates online prime discovery. The Python source advances both clock hands and generates the valuation readout from each clock's minutes address and local history.}
\label{fig:clockwork-source}
\end{figure}

\clearpage
\section{Complex-exponential picture of the seconds counters}
\label{app:complex}

A natural starting point for asking about a continuous counterpart of the prime clocks is the complex exponential. The continuous picture has one distinguished component that is not itself a prime clock: the \emph{unit clock}
\begin{equation}
z_1(\tau)=\exp(2\pi\mathrm{i}\,\tau),\qquad \tau\in\mathbb R.
\label{eq:unit-clock}
\end{equation}
Its crossings of the positive real reference point occur exactly at the integer times. The unit clock therefore supplies the event times at which the recursive prime clockwork inspects the already installed prime clocks.

For each discovered prime $p$, embed the maintained seconds counter continuously by
\begin{equation}
z_p(\tau)=p\exp\!\left(\frac{2\pi\mathrm{i}\,\tau}{p}\right),\qquad \tau\in\mathbb R.
\label{eq:complex-clock}
\end{equation}
At an integer time $n$, the point $z_p(n)$ records the same cyclic state as $R_p(n)$ and
\[
z_p(n)=p\quad\Longleftrightarrow\quad p\mid n.
\]
Thus, for $n>1$, the recursive prime-discovery condition has the continuous event form
\begin{equation}
\begin{aligned}
\text{$n$ is a prime-discovery event}
\quad\Longleftrightarrow\quad & z_1(n)=1,\\
& z_p(n)\ne p\quad\text{for every installed prime }p<n.
\end{aligned}
\label{eq:continuous-prime-event}
\end{equation}
After this event the new $n$-clock is installed at its own positive-real reference point. In other words, the discrete rule ``grow when no clock rings'' becomes an absence-of-synchronization test among autonomous continuous rotations sampled whenever the unit clock completes a turn.

Differentiation gives
\[
z'_p(\tau)=2\pi\mathrm{i}\exp\!\left(\frac{2\pi\mathrm{i}\,\tau}{p}\right),\qquad |z'_p(\tau)|=2\pi.
\]
Thus, if the radius is chosen proportional to the period, all prime-clock points have the same tangential speed while larger prime clocks have lower angular velocity.

Figure~\ref{fig:complex-clockwork} juxtaposes the planar event picture with its lifted simultaneous-state realization.

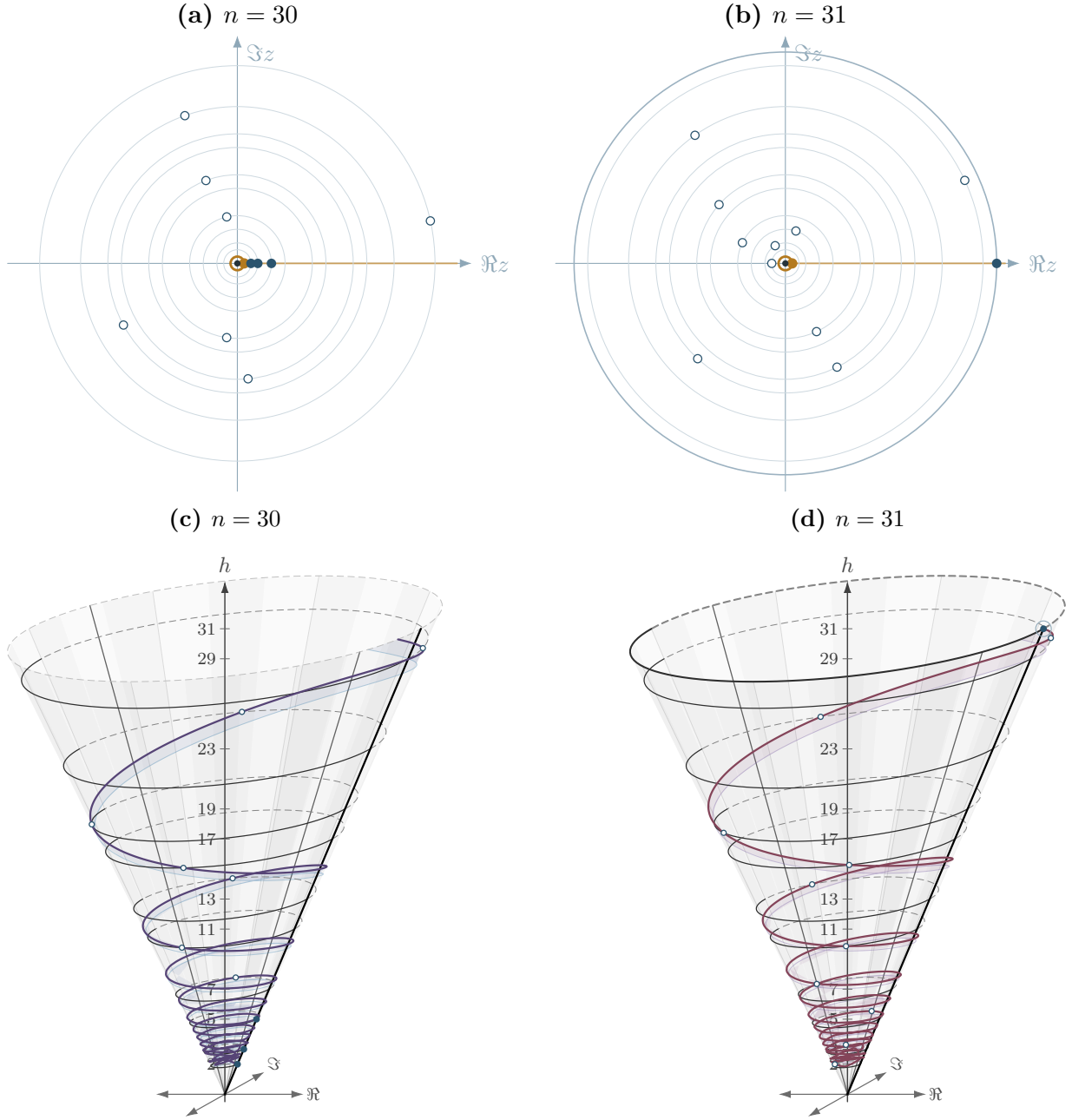
\begin{figure}[p]
\centering

\resizebox{0.96\textwidth}{!}{%
\begin{tikzpicture}[font=\small,>=Latex,line cap=round,line join=round]
\def\PCPlanarR{3.20}

\newcommand{\PCPlanarPrime}[4]{%
  \pgfmathsetmacro{\PCPR}{\PCPlanarR*(#3)/31}%
  \pgfmathsetmacro{\PCPAng}{360*(#2)/(#3)}%
  \pgfmathtruncatemacro{\PCPRem}{mod(#2,#3)}%
  \ifnum\PCPRem=0
    \fill[ClockBlue!88!black]
      ({(#1)+\PCPR*cos(\PCPAng)},{\PCPR*sin(\PCPAng)})
      circle[radius=.72mm];
  \else
    \filldraw[fill=white,draw=ClockBlue!88!black,line width=.48pt]
      ({(#1)+\PCPR*cos(\PCPAng)},{\PCPR*sin(\PCPAng)})
      circle[radius=.62mm];
  \fi
}

\newcommand{\PCPlanarPanel}[4]{%
  \node[font=\bfseries] at (#1,3.74) {#4};
  \draw[-{Latex[length=1.8mm]},ClockBlue!55]
    ({(#1)-3.48},0)--({(#1)+3.55},0) node[right] {$\Re z$};
  \draw[-{Latex[length=1.8mm]},ClockBlue!55]
    (#1,-3.45)--(#1,3.45) node[below right] {$\Im z$};
  \draw[ClockGold!72,line width=.68pt] (#1,0)--({(#1)+3.33},0);
  \draw[ClockGold,line width=1.05pt]
    (#1,0) circle[radius={\PCPlanarR/31}];
  \foreach \p in {2,3,5,7,11,13,17,19,23,29}{%
    \draw[ClockBlue!24,line width=.36pt]
      (#1,0) circle[radius={\PCPlanarR*\p/31}];
  }
  \ifnum#3=1
    \draw[ClockBlue!48,line width=.60pt]
      (#1,0) circle[radius=\PCPlanarR];
  \fi
  \fill[ClockGold] ({(#1)+\PCPlanarR/31},0) circle[radius=.78mm];
  \foreach \p in {2,3,5,7,11,13,17,19,23,29}{%
    \PCPlanarPrime{#1}{#2}{\p}{0}%
  }
  \ifnum#3=1 \PCPlanarPrime{#1}{#2}{31}{1}\fi
  \fill[Ink] (#1,0) circle[radius=.48mm];
}

\PCPlanarPanel{-4.15}{30}{0}{\textup{(a)} $n=30$}
\PCPlanarPanel{ 4.15}{31}{1}{\textup{(b)} $n=31$}
\end{tikzpicture}%
}

\vspace{1mm}

\resizebox{0.96\textwidth}{!}{%
\begin{tikzpicture}[
  line cap=round,
  line join=round,
  >={Latex[length=2.0mm,width=1.25mm]},
  every node/.style={font=\sffamily\small},
  prime label/.style={font=\sffamily\scriptsize,inner sep=1pt}
]
\def\PCH{7.55}
\def\PCR{3.18}
\def\PCViewX{0.48}
\def\PCViewY{0.27}

\newcommand{\PCPoint}[5]{%
  \pgfmathsetmacro{\PCRad}{\PCR*(#3)/31}%
  \pgfmathsetmacro{\PCX}{(#1)+\PCRad*cos(#4)+\PCViewX*\PCRad*sin(#4)}%
  \pgfmathsetmacro{\PCY}{(#2)+\PCH*(#3)/31+\PCViewY*\PCRad*sin(#4)}%
  #5%
}

\newcommand{\PCBaseAxes}[2]{%
  \draw[<->,black!58,line width=.48pt]
    ({(#1)-1.12},{#2}) -- ({(#1)+1.28},{#2})
    node[prime label,right,text=black!72] {$\Re$};
  \draw[<->,black!58,line width=.48pt]
    ({(#1)-1.34*\PCViewX},{(#2)-1.34*\PCViewY})
    -- ({(#1)+1.34*\PCViewX},{(#2)+1.34*\PCViewY})
    node[prime label,above right,text=black!72] {$\Im$};
  \fill[black!72] (#1,#2) circle[radius=.75pt];
}

\newcommand{\PCConeSurface}[2]{%
  \foreach \a [evaluate=\a as \b using \a+15] in {0,15,...,345}{%
    \pgfmathtruncatemacro{\PCShade}{mod(\a/15,2)}%
    \PCPoint{#1}{#2}{31}{\a}{\coordinate (PCA) at (\PCX,\PCY);}%
    \PCPoint{#1}{#2}{31}{\b}{\coordinate (PCB) at (\PCX,\PCY);}%
    \ifnum\PCShade=0
      \fill[black,opacity=.022] (#1,#2)--(PCA)--(PCB)--cycle;
    \else
      \fill[black,opacity=.040] (#1,#2)--(PCA)--(PCB)--cycle;
    \fi
  }
  \foreach \a in {0,45,...,315}{%
    \PCPoint{#1}{#2}{31}{\a}{%
      \draw[black,opacity=.12,line width=.28pt] (#1,#2)--(\PCX,\PCY);}%
  }
  \PCPoint{#1}{#2}{31}{154}{\coordinate (PCL) at (\PCX,\PCY);}%
  \PCPoint{#1}{#2}{31}{334}{\coordinate (PCRight) at (\PCX,\PCY);}%
  \draw[black!62,line width=.52pt] (#1,#2)--(PCL);
  \draw[black!62,line width=.52pt] (#1,#2)--(PCRight);
}

\newcommand{\PCPrimeRing}[4][]{%
  \draw[black!48,densely dashed,line width=.34pt,#1,
        domain=0:180,samples=55,variable=\a]
    plot ({(#2)+\PCR*(#4)/31*cos(\a)+\PCViewX*\PCR*(#4)/31*sin(\a)},
          {(#3)+\PCH*(#4)/31+\PCViewY*\PCR*(#4)/31*sin(\a)});
  \draw[black!82,line width=.50pt,#1,
        domain=180:360,samples=55,variable=\a]
    plot ({(#2)+\PCR*(#4)/31*cos(\a)+\PCViewX*\PCR*(#4)/31*sin(\a)},
          {(#3)+\PCH*(#4)/31+\PCViewY*\PCR*(#4)/31*sin(\a)});
}

\newcommand{\PCTopRim}[2]{%
  \draw[black!25,densely dashed,line width=.32pt,
        domain=0:360,samples=90,variable=\a]
    plot ({(#1)+\PCR*cos(\a)+\PCViewX*\PCR*sin(\a)},
          {(#2)+\PCH+\PCViewY*\PCR*sin(\a)});
}

\newcommand{\PCStateRibbon}[6]{%
  \foreach \q [evaluate=\q as \qq using \q+.5] in {2,2.5,...,30.5}{%
    \PCPoint{#1}{#2}{\q}{360*(#3)/\q}{\coordinate (PCRA) at (\PCX,\PCY);}%
    \PCPoint{#1}{#2}{\q}{360*(#4)/\q}{\coordinate (PCRB) at (\PCX,\PCY);}%
    \PCPoint{#1}{#2}{\qq}{360*(#4)/\qq}{\coordinate (PCRC) at (\PCX,\PCY);}%
    \PCPoint{#1}{#2}{\qq}{360*(#3)/\qq}{\coordinate (PCRD) at (\PCX,\PCY);}%
    \fill[#5!50!#6,opacity=.105] (PCRA)--(PCRB)--(PCRC)--(PCRD)--cycle;
  }
  \draw[#5,opacity=.36,line width=.42pt,
        domain=2:31,samples=260,variable=\q,smooth]
    plot ({(#1)+\PCR*\q/31*cos(360*(#3)/\q)
                 +\PCViewX*\PCR*\q/31*sin(360*(#3)/\q)},
          {(#2)+\PCH*\q/31
                 +\PCViewY*\PCR*\q/31*sin(360*(#3)/\q)});
}

\newcommand{\PCStateCurve}[4]{%
  \draw[#4!72!black,line width=.90pt,
        domain=2:31,samples=310,variable=\q,smooth]
    plot ({(#1)+\PCR*\q/31*cos(360*(#3)/\q)
                 +\PCViewX*\PCR*\q/31*sin(360*(#3)/\q)},
          {(#2)+\PCH*\q/31
                 +\PCViewY*\PCR*\q/31*sin(360*(#3)/\q)});
}

\newcommand{\PCStateBead}[5]{%
  \pgfmathsetmacro{\PCAng}{360*(#3)/(#4)}%
  \PCPoint{#1}{#2}{#4}{\PCAng}{%
    \ifnum#5=1
      \fill[ClockBlue!88!black] (\PCX,\PCY) circle[radius=1.55pt];
    \else
      \filldraw[fill=white,draw=ClockBlue!88!black,line width=.52pt]
        (\PCX,\PCY) circle[radius=1.18pt];
    \fi
  }%
}

\newcommand{\PCHeightTick}[3]{%
  \pgfmathsetmacro{\PCTickY}{(#2)+\PCH*(#3)/31}%
  \draw[black!60,line width=.35pt]
    ({(#1)-.075},\PCTickY)--({(#1)+.075},\PCTickY);
  \node[prime label,anchor=east,text=black!78]
    at ({(#1)-.105},\PCTickY) {$#3$};
}

\newcommand{\PCPanel}[7]{%
  \PCBaseAxes{#1}{#2}
  \PCConeSurface{#1}{#2}
  \pgfmathtruncatemacro{\PCPrevious}{#3-1}%
  \PCStateRibbon{#1}{#2}{\PCPrevious}{#3}{#6}{#7}
  \PCTopRim{#1}{#2}
  \foreach \p in {2,3,5,7,11,13,17,19,23,29}{\PCPrimeRing{#1}{#2}{\p}}
  \ifnum#5=1 \PCPrimeRing[line width=.82pt]{#1}{#2}{31}\fi
  \draw[->,black!75,line width=.58pt]
    (#1,{(#2)-.12})--(#1,{(#2)+\PCH+.78}) node[above] {$h$};
  \foreach \p in {2,3,5,7,11,13,17,19,23,29,31}{\PCHeightTick{#1}{#2}{\p}}
  \PCPoint{#1}{#2}{31}{0}{\coordinate (PCG) at (\PCX,\PCY);}%
  \draw[black,line width=.90pt] (#1,#2)--(PCG);
  \PCStateCurve{#1}{#2}{#3}{#7}
  \node[font=\bfseries\normalsize,anchor=south]
    at ({#1},{(#2)+\PCH+1.43}) {#4};
}

\PCPanel{-5.05}{0}{30}{\textup{(c)} $n=30$}{0}{StepBlue}{StepViolet}
\PCStateBead{-5.05}{0}{30}{2}{1}
\PCStateBead{-5.05}{0}{30}{3}{1}
\PCStateBead{-5.05}{0}{30}{5}{1}
\foreach \p in {7,11,13,17,19,23,29}{\PCStateBead{-5.05}{0}{30}{\p}{0}}

\PCPanel{5.05}{0}{31}{\textup{(d)} $n=31$}{1}{StepViolet}{StepRose}
\foreach \p in {2,3,5,7,11,13,17,19,23,29}{\PCStateBead{5.05}{0}{31}{\p}{0}}
\PCStateBead{5.05}{0}{31}{31}{1}
\PCPoint{5.05}{0}{31}{0}{\coordinate (PCBirth) at (\PCX,\PCY);}%
\draw[ClockBlue!55,line width=.42pt] (PCBirth) circle[radius=3.6pt];
\end{tikzpicture}%
}

\caption{Continuous complex-exponential realization of the maintained seconds layer. \textup{(a)--(b)} The planar views show the integer events $n=30$ and $n=31$.  The orange unit clock marks the integer sampling event.  Filled dark-blue discs lie on the positive-real reference ray and open dark-blue discs lie away from it.  Thus the filled prime discs $p=2,3,5$ in \textup{(a)} record exactly the prime divisors of $30$, whereas in \textup{(b)} no previously installed prime clock rings and the new $p=31$ clock is installed.  \textup{(c)--(d)} For $h>0$, let
$\Gamma_n(h)=h\exp(2\pi\mathrm{i}n/h)$.  The lifted point
$(\Re\Gamma_n(h),\Im\Gamma_n(h),h)$ lies on the transparent ambient cone
$x^2+y^2=h^2$, and its prime-height section satisfies
$\Gamma_n(p)=z_p(n)$.  The faint ribbon swept from $\Gamma_{n-1}$ to
$\Gamma_n$ records one integer step.  At $n=30$ the filled beads at
$p=2,3,5$ lie on the positive-real generator, exactly recording the prime
divisors of $30$.  At $n=31$ no earlier prime bead is aligned, and the new
$p=31$ section is installed.  Open beads show the remaining prime-clock
phases.  The cone and ribbon visualize the
continuous seconds dynamics. Completed turns belong to the minutes counters;
the valuation readouts are not an additional geometric layer.}
\label{fig:complex-clockwork}
\end{figure}

\begin{remark}[Continuous visualization of the maintained cyclic layer]
The complex exponential interpolates only the seconds motion. Its completed turns correspond to the integer minutes reading $M_p$. Neither the cone nor the ribbon displays the valuation $V_p$. At an integer ring, $V_p$ is recovered from the current minutes address and the clock's own valuation history by~\eqref{eq:valuation-recursion}. Thus the continuous picture visualizes the cyclic part of the autonomous two-hand update without conflating elapsed cycles with prime-exponent depth.
\end{remark}

\clearpage
\section{A portrait run for integers 1 through 30}
\label{app:tables}

Table~\ref{tab:exponents-30} shows the factorized output. Each prime column records the valuation readout $V_p(n)=\nu_p(n)$. Soft gray rows are reserved for primes; they mark growth events of the recursion. At such a time no previously installed prime clock rings, so a new clock is created and its prime coordinate appears for the first time.

\begin{table}[h!]
\centering
\caption{Prime-exponent output for $1\le n\le30$. Soft gray rows mark newly encountered primes.}
\label{tab:exponents-30}
\small
\setlength{\tabcolsep}{2.3pt}
\renewcommand{\arraystretch}{1.16}
\begin{tabular}{@{}>{\bfseries\raggedleft\arraybackslash}p{5mm} >{\raggedright\arraybackslash}p{28mm} *{10}{>{\centering\arraybackslash}p{9.0mm}}@{}}
\toprule
$n$ & factorized form & $2$ & $3$ & $5$ & $7$ & $11$ & $13$ & $17$ & $19$ & $23$ & $29$\\
\midrule
1 & $1$ & 0 & 0 & 0 & 0 & 0 & 0 & 0 & 0 & 0 & 0 \\
\rowcolor{PrimeGray}
2 & $2$ & 1 & 0 & 0 & 0 & 0 & 0 & 0 & 0 & 0 & 0 \\
\rowcolor{PrimeGray}
3 & $3$ & 0 & 1 & 0 & 0 & 0 & 0 & 0 & 0 & 0 & 0 \\
4 & $2^{2}$ & 2 & 0 & 0 & 0 & 0 & 0 & 0 & 0 & 0 & 0 \\
\rowcolor{PrimeGray}
5 & $5$ & 0 & 0 & 1 & 0 & 0 & 0 & 0 & 0 & 0 & 0 \\
6 & $2\cdot 3$ & 1 & 1 & 0 & 0 & 0 & 0 & 0 & 0 & 0 & 0 \\
\rowcolor{PrimeGray}
7 & $7$ & 0 & 0 & 0 & 1 & 0 & 0 & 0 & 0 & 0 & 0 \\
8 & $2^{3}$ & 3 & 0 & 0 & 0 & 0 & 0 & 0 & 0 & 0 & 0 \\
9 & $3^{2}$ & 0 & 2 & 0 & 0 & 0 & 0 & 0 & 0 & 0 & 0 \\
10 & $2\cdot 5$ & 1 & 0 & 1 & 0 & 0 & 0 & 0 & 0 & 0 & 0 \\
\rowcolor{PrimeGray}
11 & $11$ & 0 & 0 & 0 & 0 & 1 & 0 & 0 & 0 & 0 & 0 \\
12 & $2^{2}\cdot 3$ & 2 & 1 & 0 & 0 & 0 & 0 & 0 & 0 & 0 & 0 \\
\rowcolor{PrimeGray}
13 & $13$ & 0 & 0 & 0 & 0 & 0 & 1 & 0 & 0 & 0 & 0 \\
14 & $2\cdot 7$ & 1 & 0 & 0 & 1 & 0 & 0 & 0 & 0 & 0 & 0 \\
15 & $3\cdot 5$ & 0 & 1 & 1 & 0 & 0 & 0 & 0 & 0 & 0 & 0 \\
16 & $2^{4}$ & 4 & 0 & 0 & 0 & 0 & 0 & 0 & 0 & 0 & 0 \\
\rowcolor{PrimeGray}
17 & $17$ & 0 & 0 & 0 & 0 & 0 & 0 & 1 & 0 & 0 & 0 \\
18 & $2\cdot 3^{2}$ & 1 & 2 & 0 & 0 & 0 & 0 & 0 & 0 & 0 & 0 \\
\rowcolor{PrimeGray}
19 & $19$ & 0 & 0 & 0 & 0 & 0 & 0 & 0 & 1 & 0 & 0 \\
20 & $2^{2}\cdot 5$ & 2 & 0 & 1 & 0 & 0 & 0 & 0 & 0 & 0 & 0 \\
21 & $3\cdot 7$ & 0 & 1 & 0 & 1 & 0 & 0 & 0 & 0 & 0 & 0 \\
22 & $2\cdot 11$ & 1 & 0 & 0 & 0 & 1 & 0 & 0 & 0 & 0 & 0 \\
\rowcolor{PrimeGray}
23 & $23$ & 0 & 0 & 0 & 0 & 0 & 0 & 0 & 0 & 1 & 0 \\
24 & $2^{3}\cdot 3$ & 3 & 1 & 0 & 0 & 0 & 0 & 0 & 0 & 0 & 0 \\
25 & $5^{2}$ & 0 & 0 & 2 & 0 & 0 & 0 & 0 & 0 & 0 & 0 \\
26 & $2\cdot 13$ & 1 & 0 & 0 & 0 & 0 & 1 & 0 & 0 & 0 & 0 \\
27 & $3^{3}$ & 0 & 3 & 0 & 0 & 0 & 0 & 0 & 0 & 0 & 0 \\
28 & $2^{2}\cdot 7$ & 2 & 0 & 0 & 1 & 0 & 0 & 0 & 0 & 0 & 0 \\
\rowcolor{PrimeGray}
29 & $29$ & 0 & 0 & 0 & 0 & 0 & 0 & 0 & 0 & 0 & 1 \\
30 & $2\cdot 3\cdot 5$ & 1 & 1 & 1 & 0 & 0 & 0 & 0 & 0 & 0 & 0 \\
\bottomrule
\end{tabular}
\end{table}

\clearpage
Table~\ref{tab:residues-30} gives the simultaneous residue-clock state for the same run.  For compact comparison, all prime columns through $29$ are displayed in every row, even before the corresponding clock would have been installed by the dynamic recursion. At time $n$, only columns with $p\le n$ are active. Reading across an active part of a row gives the joint state at time $n$; reading down a column shows one autonomous prime clock. Every soft gray row is a growth event. Among the clocks active just before that row, none rings, and the newly created $p=n$ clock begins at zero.

\begin{table}[h!]
\centering
\caption{Prime-residue states for $1\le n\le30$. A zero is a ring of that prime clock; soft gray rows mark newly encountered primes.}
\label{tab:residues-30}
\small
\setlength{\tabcolsep}{2.0pt}
\renewcommand{\arraystretch}{1.16}
\begin{tabular}{@{}>{\bfseries\raggedleft\arraybackslash}p{6mm} *{10}{>{\centering\arraybackslash}p{10.8mm}}@{}}
\toprule
$n$ & $R_2$ & $R_3$ & $R_5$ & $R_7$ & $R_{11}$ & $R_{13}$ & $R_{17}$ & $R_{19}$ & $R_{23}$ & $R_{29}$\\
\midrule
1 & 1 & 1 & 1 & 1 & 1 & 1 & 1 & 1 & 1 & 1 \\
\rowcolor{PrimeGray}
2 & 0 & 2 & 2 & 2 & 2 & 2 & 2 & 2 & 2 & 2 \\
\rowcolor{PrimeGray}
3 & 1 & 0 & 3 & 3 & 3 & 3 & 3 & 3 & 3 & 3 \\
4 & 0 & 1 & 4 & 4 & 4 & 4 & 4 & 4 & 4 & 4 \\
\rowcolor{PrimeGray}
5 & 1 & 2 & 0 & 5 & 5 & 5 & 5 & 5 & 5 & 5 \\
6 & 0 & 0 & 1 & 6 & 6 & 6 & 6 & 6 & 6 & 6 \\
\rowcolor{PrimeGray}
7 & 1 & 1 & 2 & 0 & 7 & 7 & 7 & 7 & 7 & 7 \\
8 & 0 & 2 & 3 & 1 & 8 & 8 & 8 & 8 & 8 & 8 \\
9 & 1 & 0 & 4 & 2 & 9 & 9 & 9 & 9 & 9 & 9 \\
10 & 0 & 1 & 0 & 3 & 10 & 10 & 10 & 10 & 10 & 10 \\
\rowcolor{PrimeGray}
11 & 1 & 2 & 1 & 4 & 0 & 11 & 11 & 11 & 11 & 11 \\
12 & 0 & 0 & 2 & 5 & 1 & 12 & 12 & 12 & 12 & 12 \\
\rowcolor{PrimeGray}
13 & 1 & 1 & 3 & 6 & 2 & 0 & 13 & 13 & 13 & 13 \\
14 & 0 & 2 & 4 & 0 & 3 & 1 & 14 & 14 & 14 & 14 \\
15 & 1 & 0 & 0 & 1 & 4 & 2 & 15 & 15 & 15 & 15 \\
16 & 0 & 1 & 1 & 2 & 5 & 3 & 16 & 16 & 16 & 16 \\
\rowcolor{PrimeGray}
17 & 1 & 2 & 2 & 3 & 6 & 4 & 0 & 17 & 17 & 17 \\
18 & 0 & 0 & 3 & 4 & 7 & 5 & 1 & 18 & 18 & 18 \\
\rowcolor{PrimeGray}
19 & 1 & 1 & 4 & 5 & 8 & 6 & 2 & 0 & 19 & 19 \\
20 & 0 & 2 & 0 & 6 & 9 & 7 & 3 & 1 & 20 & 20 \\
21 & 1 & 0 & 1 & 0 & 10 & 8 & 4 & 2 & 21 & 21 \\
22 & 0 & 1 & 2 & 1 & 0 & 9 & 5 & 3 & 22 & 22 \\
\rowcolor{PrimeGray}
23 & 1 & 2 & 3 & 2 & 1 & 10 & 6 & 4 & 0 & 23 \\
24 & 0 & 0 & 4 & 3 & 2 & 11 & 7 & 5 & 1 & 24 \\
25 & 1 & 1 & 0 & 4 & 3 & 12 & 8 & 6 & 2 & 25 \\
26 & 0 & 2 & 1 & 5 & 4 & 0 & 9 & 7 & 3 & 26 \\
27 & 1 & 0 & 2 & 6 & 5 & 1 & 10 & 8 & 4 & 27 \\
28 & 0 & 1 & 3 & 0 & 6 & 2 & 11 & 9 & 5 & 28 \\
\rowcolor{PrimeGray}
29 & 1 & 2 & 4 & 1 & 7 & 3 & 12 & 10 & 6 & 0 \\
30 & 0 & 0 & 0 & 2 & 8 & 4 & 13 & 11 & 7 & 1 \\
\bottomrule
\end{tabular}
\end{table}

\end{document}